\documentclass[11pt]{amsart}
\usepackage{amsthm,amsmath,amssymb}
\usepackage{stmaryrd,mathrsfs}
\usepackage{enumerate}
\usepackage[T1]{fontenc}
\usepackage{esint}
\usepackage{tikz}
\usepackage{hyperref}
\usepackage{stackengine}
\usepackage{verbatim}
\stackMath
\allowdisplaybreaks
\usepackage{geometry}
\usepackage{hyperref}
	\hypersetup{colorlinks, breaklinks,
            	linkcolor = blue,
				urlcolor = blue,
				anchorcolor = blue,
				citecolor = blue}

\usepackage[doi=true,isbn=true,sorting=nty,giveninits=true,maxbibnames=99,backend=biber]{biblatex}
\newcommand{\abs}[1]{|#1|}

\newcommand{\Norm}[2]{\|#1\|_{#2}}

\newcommand{\ave}[1]{\langle #1\rangle}
\newcommand{\BMO}[0]{\operatorname{BMO}}
\newcommand{\CMO}[0]{\operatorname{CMO}}

\newcommand{\supp}[0]{\operatorname{supp}}

\DeclareMathOperator*{\esssup}{ess\;sup}
\DeclareMathOperator*{\essinf}{ess\;inf}

\newcommand{\R}{\mathbb{R}}
\newcommand{\C}{\mathbb{C}}
\newcommand{\N}{\mathbb{N}}
\newcommand{\Z}{\mathbb{Z}}

\newcommand{\dd}{\,\mathrm{d}}

\swapnumbers \numberwithin{equation}{section}

\theoremstyle{plain}
\newtheorem{theorem}[equation]{Theorem}
\newtheorem{proposition}[equation]{Proposition}

\newtheorem{lemma}[equation]{Lemma}

\theoremstyle{definition}
\newtheorem{definition}[equation]{Definition}

\newtheorem{remark}[equation]{Remark}

\makeatletter
\@namedef{subjclassname@2020}{
  \textup{2020} Mathematics Subject Classification}
 \def\@textbottom{\vskip \z@ \@plus 1pt}
 \let\@texttop\relax
\makeatother

\begin{document}

\title[Extrapolation of compactness]{Extrapolation for bilinear compact operators in the variable exponent setting}

\author{Spyridon Kakaroumpas}
\address{Institut f\"{u}r Mathematik, Julius-Maximilians-Universit\"{a}t
W\"{u}rzburg, Emil-Fischer-Strasse 41, 97074 W\"{u}rzburg, Germany}
\email{spyridon.kakaroumpas@uni-wuerzburg.de}

\author{Stefanos Lappas}
\address{Department of Mathematical Analysis, Faculty of Mathematics and Physics, Charles University, Sokolovsk\'a 83, 186 75 Praha 8, Czech Republic}
\email{stefanos.lappas@matfyz.cuni.cz; vlappas@hotmail.com}

\thanks{S. L. was supported by the Primus research programme PRIMUS/21/SCI/002 of Charles University and the Foundation for Education and European Culture, founded by Nicos and Lydia Tricha.}





\begin{abstract}
We establish extrapolation of compactness for bilinear operators in the scale of weighted variable exponent Lebesgue spaces. First, we prove an abstract principle relying on the Cobos--Fern\'{a}ndez-Cabrera--Mart\'{i}nez theorem. Then, as an application we deduce new compactness results for the commutators of bilinear $\omega$-Calder\'{o}n--Zygmund operators, bilinear fractional integrals and bilinear Fourier multipliers acting on weighted variable exponent Lebesgue spaces. Our work extends and unifies among others earlier works of the second named author together with Hyt\"{o}nen as well as Oikari.
\end{abstract}

\maketitle


\section{Introduction}

This paper brings together several areas: weighted estimates for singular integrals, multilinear operators, variable exponent Lebesgue spaces, and extrapolation of boundedness, respectively compactness. Let us briefly recall the history of each of the involved areas.

It is a classical result that if $T$ is a Calder\'{o}n--Zygmund operator on $\R^n$, $1<p<\infty$, and $w$ is a weight (meaning an a.e.~positive function) on $\R^n$, such that the so called Muckenhoupt $A_p$ characteristic
\begin{equation*}
    [w]_{A_p} := \sup_{Q}\left(\frac{1}{|Q|}\int_{Q}w(x)^{p}\,\mathrm{d}x\right)^{1/p}\left(\frac{1}{|Q|}\int_{Q}w(x)^{-p'}\,\mathrm{d}x\right)^{1/p'}
\end{equation*}
is finite, where the supremum is taken over all cubes $Q\subset\R^n$ (with faces parallel to the coordinate hyperplanes) and $\frac{1}{p}+\frac{1}{p'}=1$, then $T$ acts boundedly from $L^{p}(w)$ into itself. Here we denote \begin{equation*}
    \Vert f\Vert_{L^{p}(w)} := \left(\int_{\R^n}|f(x)w(x)|^{p}\,\mathrm{d}x\right)^{1/p}.
\end{equation*}
Such boundedness results and quantitative versions of them are often proved more easily for $p=2$, among other reasons, because $L^2(w)$ is a Hilbert space. To extend them to the whole range $1<p<\infty$, one needs an \emph{extrapolation} result. In this case, the relevant method was introduced by Rubio de Francia \cite{Rubio_de_Francia_1984}. Several works followed further developing this method, until Duoandikoetxea \cite{Duoandikoetxea_2011} proved a sharp version of it. We refer to the book \cite{Book_Extrapolation} for a thorough historical overview as well as detailed proofs and generalizations.

There exist several directions for studying more general problems than the above mentioned boundedness. One consists in considering multilinear Calder\'{o}n--Zygmund operators. Lerner, Ombrosi, P\'erez, Torres and Trujillo-Gonz\'alez \cite{LOPTT09} showed that if $T$ is a $m$-linear Calder\'{o}n--Zygmund operator on $\R^n$, $\vec{p}=(p_1,\ldots,p_m)$ is a $m$-tuple of exponents from $(1,\infty)$, and $\vec{w}=(w_1,\ldots,w_m)$ is a $m$-tuple of weights such that
\begin{equation*}
    [\vec{w}]_{A_{\vec{p}}} :=
    \sup_{Q}\left(\frac{1}{|Q|}\int_{Q}\prod_{j=1}^{m}w_j(x)^{p}\,\mathrm{d}x\right)^{1/p}\prod_{j=1}^{m}\left(\frac{1}{|Q|}\int_{Q}w_j(x)^{-p_j'}\,\mathrm{d}x\right)^{1/p_j'}
\end{equation*}
is finite, where
\begin{equation*}
    \frac{1}{p} = \sum_{j=1}^{m}\frac{1}{p_j},
\end{equation*}
then $T$ maps $L^{p_1}(w_1)\times\cdots\times L^{p_m}(w_m)$ into $L^{p}(\prod_{j=1}^{m}w_j)$ boundedly. Extrapolation results of such boundedness have played here also a key role in the further development of the theory, see \cite{Nieraeth_2019}.

Another direction consists in considering weighted Lebesgue spaces $L^{p(\cdot)}(w)$ with a variable exponent $p(\cdot)$, that means $p(\cdot)$ is now a not necessarily constant function $p(\cdot):\R^n\to[0,\infty]$. For the precise definition of these spaces we refer to Section~\ref{sec: preliminaries}. An exhaustive presentation of the unweighted variable exponent Lebesgue spaces $L^{p(\cdot)}(\R^n)$ can be found in \cite{CF2013, DHHR2011}. Cruz-Uribe, Fiorenza and Neugebauer \cite{Cruz-Uribe_Firorenza_Neugebauer_2003} showed that for every $p(\cdot)\in \mathscr{P}\cap \mathrm{LH}$ (see Section~\ref{sec: preliminaries} for the notation), the Hardy--Littlewood maximal operator is bounded on $L^{p(\cdot)}(w)$ if any only if
\begin{equation*}
    [w]_{\mathcal{A}_{p(\cdot)}} := \sup_{Q} |Q|^{-1}\Vert w\chi_{Q}\Vert_{p(\cdot)}\Vert w^{-1}\chi_{Q}\Vert_{p'(\cdot)}
\end{equation*}
is finite. This was extended in \cite{CG2020} to the bilinear setting, where moreover analogous boundedness results for bilinear singular integrals were proved. Also, there is by now a rich extrapolation theory for variable exponent spaces both in the linear and the multilinear setting, see \cite{CW2017, Cao2022, Nieraeth_2023}. However, it is still unknown if one can prove the exact linear and multilinear analogue of Rubio de Francia's extrapolation of boundedness on weighted Lebesgue spaces with variable exponents (see also \cite[Remark 1.7]{LO2025}).    

Over the last few years, the extension of extrapolation results to compact operators has drawn the attention of many authors. In particular, the first work that dealt with this mathematical problem in the linear setting was due to Hyt\"{o}nen and the second named author \cite{HL2023}. Further extensions of this result in the setting of weighted Morrey spaces, two weight problems, Banach function spaces and one sided situations can be found in \cite{Lappas2022}, \cite{Liu2022}, \cite{Lorist2024} (see also \cite{LO2025} for an alternative proof in the context of weighted Lebesgue spaces with variable exponents) and \cite{MartinReyes2025}, respectively. Moreover, first Cao, Olivo and Yabuta \cite{CaoOlivoYabuta2022} established extrapolation results for multilinear compact operators (see also \cite{HL2022}).

\subsection*{Main results}

Motivated by the previously mentioned works, in this paper we establish the following abstract diagonal and off-diagonal extrapolation, full range and limited range theorems for bilinear compact operators in the setting of weighted variable exponent Lebesgue spaces (see Section \ref{sec: preliminaries} for detailed definitions and explanation of the notation). 

\begin{theorem}\label{thm:main_result}
Let $T$ be a bilinear operator such that there exist fixed constants $t\in(0,\infty)$ and $\gamma\in\left[0,\infty\right)$ satisfying both of the following two assumptions.
\begin{enumerate}
\item\label{eq:main1} For all proper $2$-admissible quadruples $(\vec{p}_0(\cdot),q_0(\cdot),\vec{1},\infty)$ with
\begin{equation*}
    (q_0)_{-} > 1,
\end{equation*}
\begin{equation*}
    t<\min\{(p_{0,j})_{-}:~j=1,2\}
\end{equation*}
and
\begin{equation*}
    \frac{1}{p_0(\cdot)}-\frac{1}{q_0(\cdot)}=\gamma,
\end{equation*}
the following holds. Whenever $\vec{w}_0$ is a $2$-tuple of weights such that $\vec{w}_0\in\mathcal{A}_{\vec{p}_0(\cdot),q_0(\cdot)}$ and $\vec{w}_0^{t}\in\mathcal{A}_{\frac{\vec{p}_0(\cdot)}{t},\frac{q_0(\cdot)}{t}}$, we have
\begin{equation*}
T:L^{p_{0,1}(\cdot)}(w_{0,1})\times L^{p_{0,2}(\cdot)}(w_{0,2})\to L^{q_0(\cdot)}(\nu_{\vec{w}_0})
\end{equation*}
boundedly.

\item\label{eq:main2} There exist some fixed proper $2$-admissible quadruple $(\vec{p}_1(\cdot), q_1(\cdot), \vec{1},\infty)$ with
\begin{equation*}
    (q_1)_{-} > 1,
\end{equation*}
\begin{equation*}
    t<\min\{(p_{1,j})_{-}:~j=1,2\}
\end{equation*}
and
\begin{equation*}
    \frac{1}{p_1(\cdot)}-\frac{1}{q_1(\cdot)}=\gamma,
\end{equation*}
and a fixed $2$-tuple of weights $\vec{w}_1$ with $\vec{w}_1\in\mathcal{A}_{\vec{p}_1(\cdot),q_1(\cdot)}$ and $\vec{w}_1^{t}\in\mathcal{A}_{\frac{\vec{p}_1(\cdot)}{t},\frac{q_1(\cdot)}{t}}$, such that
\begin{equation*}
    T:L^{p_{1,1}(\cdot)}(w_{1,1})\times L^{p_{1,2}(\cdot)}(w_{1,2})\to L^{q_1(\cdot)}(\nu_{\vec{w}_1})
\end{equation*}
        compactly.
\end{enumerate}
Then: For all proper $2$-admissible quadruples $(\vec{p}(\cdot), q(\cdot),\vec{1},\infty)$ such that
\begin{equation*}
    q_{-} > 1,
\end{equation*}
\begin{equation*}
    t<\min\{(p_{j})_{-}:~j=1,2\}
\end{equation*}
and
\begin{equation*}
    \frac{1}{p(\cdot)}-\frac{1}{q(\cdot)}=\gamma,
\end{equation*}
and for all $2$-tuples of weights $\vec{w}$ with $\vec{w}\in\mathcal{A}_{\vec{p}(\cdot),q(\cdot)}$ and $\vec{w}^{t}\in\mathcal{A}_{\frac{\vec{p}(\cdot)}{t},\frac{q(\cdot)}{t}}$, we have
\begin{equation*}
    T:L^{p_{1}(\cdot)}(w_{1})\times L^{p_{2}(\cdot)}(w_{2})\to L^{q(\cdot)}(\nu_{\vec{w}})
\end{equation*}
compactly.
\end{theorem}

\begin{theorem}\label{thm:main_result_limited_range}
Let $T$ be a bilinear operator such that there exist a fixed constant $\gamma\in\left[0,\infty\right)$, a fixed $m$-tuple $\vec{r}\in[1,\infty)^{m}$ and a fixed constant $s\in(0,\infty]$ satisfying both of the following two assumptions.
\begin{enumerate}
\item\label{eq:main1_limited_range} For all proper $2$-admissible quadruples $(\vec{p}_0(\cdot),q_0(\cdot),\vec{r},s)$ with
\begin{equation*}
    (q_0)_{-} > 1
\end{equation*}
and
\begin{equation*}
    \frac{1}{p_0(\cdot)}-\frac{1}{q_0(\cdot)}=\gamma,
\end{equation*}
the following holds. Whenever $\vec{w}_0$ is a $2$-tuple of weights such that $\vec{w}_0\in\mathcal{A}_{(\vec{p}_0(\cdot),q_0(\cdot)),(\vec{r},s)}$, we have
\begin{equation*}
    T:L^{p_{0,1}(\cdot)}(w_{0,1})\times L^{p_{0,2}(\cdot)}(w_{0,2})\to L^{q_0(\cdot)}(\nu_{\vec{w}_0})
\end{equation*}
boundedly.

\item\label{eq:main2_limited_range} There exist some fixed proper $2$-admissible quadruple $(\vec{p}_1(\cdot), q_1(\cdot), \vec{r},s)$ with
\begin{equation*}
    (q_1)_{-} > 1
\end{equation*}
and
\begin{equation*}
    \frac{1}{p_1(\cdot)}-\frac{1}{q_1(\cdot)}=\gamma,
\end{equation*}
and a fixed $2$-tuple of weights $\vec{w}_1$ with $\vec{w}_1\in\mathcal{A}_{(\vec{p}_1(\cdot),q_1(\cdot)),(\vec{r},s)}$, such that
\begin{equation*}
    T:L^{p_{1,1}(\cdot)}(w_{1,1})\times L^{p_{1,2}(\cdot)}(w_{1,2})\to L^{q_1(\cdot)}(\nu_{\vec{w}_1})
\end{equation*}
compactly.
\end{enumerate}
Then: For all proper $2$-admissible quadruples $(\vec{p}(\cdot), q(\cdot),\vec{r},s)$ such that
\begin{equation*}
    q_{-} > 1
\end{equation*}
and
\begin{equation*}
    \frac{1}{p(\cdot)}-\frac{1}{q(\cdot)}=\gamma,
\end{equation*}
and for all $2$-tuples of weights $\vec{w}$ with $\vec{w}\in\mathcal{A}_{(\vec{p}(\cdot),q(\cdot)),(\vec{r},s)}$, we have
\begin{equation*}
    T:L^{p_{1}(\cdot)}(w_{1})\times L^{p_{2}(\cdot)}(w_{2})\to L^{q(\cdot)}(\nu_{\vec{w}})
\end{equation*}
compactly.
\end{theorem}

As applications of Theorem~\ref{thm:main_result}, we obtain new compactness results for the commutators of bilinear $\omega$-Calder\'{o}n--Zygmund operators and bilinear fractional integrals (see Subsections~\ref{subsec:CZ} and~\ref{subsec:fractional}). Moreover, as an application of Theorem~\ref{thm:main_result_limited_range}, we obtain new compactness results for the commutators of bilinear Fourier multipliers (see Subsection~\ref{subsec:fourier}).

\begin{remark}\label{Rmk: main}
By assuming $t=1$ and choosing all the variable exponents to be constant in the assumptions and conclusions of Theorems~\ref{thm:main_result} and~\ref{thm:main_result_limited_range}, we can extend or recover several results of \cite{CaoOlivoYabuta2022, HL2022, Lorist2024} in the bilinear case.
\end{remark}

The remainder of this paper is organized as follows. In Section \ref{sec: preliminaries}, we recall the definition of bilinear compact operators and state several definitions about unweighted and weighted variable Lebesgue spaces. In Section \ref{sec: auxiliary results}, we gather some auxiliary lemmata that will aid us in proving our main results. As far as we know, one of these lemmata, namely, Lemma \ref{lem:Apqrs char.} is new. Section \ref{sec: factorizations} is devoted to the proof of two key factorization results in the setting of variable multilinear Muckenhoupt weights (see Lemmata \ref{lem:KeyLemma} and \ref{lem:KeyLemma_with_t}). In Section \ref{sec: Lions--Peetre comp. thm.}, we recall several known facts about complex interpolation for bilinear compact operators. Among them is the abstract result of Cobos--Fern\'{a}ndez-Cabrera--Mart\'{i}nez \cite[Theorem 3.2]{CFCM2020}. Lastly, in Sections \ref{sec: pf. main result} and \ref{sec: applications}, we provide the proofs of Theorems~\ref{thm:main_result} and~\ref{thm:main_result_limited_range} and their applications, respectively.

\subsection*{Notation}

We write $x\lesssim y$ to mean that there exists a constant $C>0$ such that $x\leq Cy$. This constant $C$ is just an absolute constant or one depending on parameters that are specified each time or that are understood by the context. Moreover, we denote by $|Q|$ the n-dimensional Lebesgue measure of a cube $Q\subset\R^n$ and $\langle w\rangle_Q$ will denote the average of $w$ over $Q$, that is, $|Q|^{-1}\int_{Q}w(x)\,\mathrm{d}x$.

\section{Preliminaries}\label{sec: preliminaries}

Throughout the paper, we mostly consider bilinear compact operators. Hence, in this section we recall this definition for the reader's convenience. Given the normed spaces $X, Y$ and $Z$, a bilinear operator $T:X\times Y\rightarrow Z$ is said to be compact if the set $\{T(x,y):\;\|x\|\leq1,\;\;\|y\|\leq1\}$ is precompact in $Z$. Some other equivalent formulations of this definition were given in \cite[Proposition 1]{BenTor2013}.

\subsection{Variable Lebesgue spaces} In this subsection, we recall some basic definitions of variable Lebesgue spaces. For more information, we refer the reader to \cite{CF2013} and \cite{DHHR2011}.

\subsubsection{Notation for variable exponent functions} Given any measurable function $p(\cdot):\R^n\to[0,\infty]$ and a set $E\subset \mathbb{R}^n$, we denote
    $$
    p_{-}(E)=\essinf_{x\in E}p(x),\qquad p_{+}(E)=\esssup_{x\in E}p(x),
    $$
and if $E=\R^n$, then we will write $p_{-}(\mathbb{R}^n) = p_{-}$ and $p_{+}(\mathbb{R}^n) = p_{+}$.

In this paper, we will consider several special classes of measurable functions $p(\cdot):\R^n\to[0,\infty]$:
\begin{equation*}
    \mathscr{P}_0:=\{p(\cdot):~0<p_{-} \leq p_{+}<\infty\},
\end{equation*}
\begin{equation*}
    \mathcal{P}:=\{p(\cdot):~1\leq p_{-}\leq p_{+}<\infty\},
\end{equation*}
and lastly
\begin{equation*}
    \mathscr{P}:=\{p(\cdot):~1<p_{-}\leq p_{+}<\infty\}.
\end{equation*}
 
Given $p(\cdot)\in\mathcal{P}$, the dual exponent $p'(\cdot)\in\mathcal{P}$ is given by 
$1/p(\cdot)+ 1/p'(\cdot)=1.$

If $0<|E|<\infty$, we define the harmonic mean $p_E$ on the set $E\subset\R^n$ by
\begin{equation*}
  \frac{1}{p_E}:=\displaystyle\stackinset{c}{}{c}{}{-\mkern4mu}{\displaystyle\int_E}\;\frac{1}{p(x)}\dd x=\frac{1}{|E|}\int_E\;\frac{1}{p(x)}\dd x.
\end{equation*}

\subsubsection{Luxemburg norm of variable Lebesgue spaces} Given $p(\cdot)\in\mathscr{P}_0$, define the modular associated with $p(\cdot)$ by 
\begin{equation}\label{eq:modular}
    \rho_{p(\cdot)}(f):=\int_{\R^n}|f(x)|^{p(x)}\dd x.
\end{equation} 
The variable Lebesgue space $L^{p(\cdot)}(\R^n)$ consists of all measurable
functions $f:\R^n\rightarrow \R$ such that 
\begin{equation*}
    \|f\|_{L^{p(\cdot)}}= \|f\|_{{p(\cdot)}}:=\inf\{\lambda\in(0,\infty):\rho_{p(\cdot)}(f/\lambda)\leq 1\}<\infty.
\end{equation*}

\subsubsection{Weighted variable Lebesgue spaces} By a weight $w$ we mean a measurable function $w$ on $\R^n$ such that $0<w(x)<\infty$ for almost every $x\in\R^n$. Given a weight $w$, we define the weighted variable Lebesgue space $L^{p(\cdot)}(w)$ through
$$
  \|f\|_{L^{p(\cdot)}(w)}:= \|fw\|_{p(\cdot)}< \infty.
$$
  
\subsubsection{Log-H\"{o}lder continuity} For technical reasons, we will be considering variable exponent functions possessing the local and asymptotic H\"older continuity properties. We recall these below.

\begin{definition}\label{logHolder}
For $p(\cdot):\R^n\rightarrow\R$, we say that
\begin{enumerate}
    \item\label{cond1} it is locally log-H{\"o}lder continuous, denoted by $p(\cdot)\in \mathrm{LH}_0,$ if there exists $C_0>0$ such that for all $x,y\in\R^n$, whenever $|x-y|<\tfrac{1}{2}$, then
    \begin{equation*}
        |p(x)-p(y)|\leq\frac{C_0}{-\log(|x-y|)};
    \end{equation*}
    \item\label{cond2} it is log-H{\"o}lder continuous at infinity, denoted by $p(\cdot)\in \mathrm{LH}_{\infty},$ if there exist $p_{\infty}\in\R$ and $C_\infty>0$ such that for all $x\in\R^n$,
    \begin{equation*}
        |p(x)-p_{\infty}|\leq\frac{C_{\infty}}{\log(e+|x|)}.
    \end{equation*}
\end{enumerate}
We define $\mathrm{LH} := \mathrm{LH}_0\cap \mathrm{LH}_{\infty}.$
\end{definition}

\subsection{Variable exponent Muckenhoupt characteristics}

Here we define the main classes of weights that we will be considering in this paper.

\begin{definition}
\noindent\begin{enumerate}
    \item[(1)] 
A quadruple $(\vec{p}(\cdot), q(\cdot), \vec{r}, s)$ is said to be \emph{$m$-admissible} if all of the following conditions are satisfied:
\begin{itemize}
    \item $\vec{r}= (r_1,\ldots,r_m)\in(0,\infty)^m$ is a $m$-tuple of constants.

    \item $s\in(0,\infty]$ is a constant.

    \item $\vec{p}(\cdot) = (p_1(\cdot),\ldots,p_m(\cdot))\in\mathscr{P}_0^m$ is a $m$-tuple of variable exponents with
    \begin{equation*}
    r_j < (p_{j})_{-},\quad\forall j=1,\ldots,m.
    \end{equation*}

    \item $q(\cdot)\in\mathscr{P}_{0}$ is a variable exponent with $q_{+} < s$.

    \item There is a constant $\gamma\in[0,\infty)$ such that
    \begin{equation*}
    \frac{1}{p(\cdot)} - \frac{1}{q(\cdot)} = \gamma,
    \end{equation*}
    where the variable exponent $p(\cdot)\in\mathscr{P}_0$ is defined via
    \begin{equation*}
    \frac{1}{p(\cdot)} := \sum_{j=1}^{m}\frac{1}{p_j(\cdot)}.
    \end{equation*}
\end{itemize}
In this case, we define the constant $r\in(0,\infty)$ via
    \begin{equation*}
    \frac{1}{r} := \sum_{j=1}^{m}\frac{1}{r_j}.
    \end{equation*}

    \item[(2)] A $m$-admissible quadruple $(\vec{p}(\cdot), q(\cdot), \vec{r}, s)$ is said to be \emph{proper} if
    \begin{equation*}
    p_j(\cdot)\in \mathrm{LH},\quad\forall j=1,\ldots,m
    \end{equation*}
    and
    \begin{equation*}
    q(\cdot)\in \mathrm{LH}.
    \end{equation*}
   Observe that in this case, since $(p_j)_+<\infty$, by \cite[Proposition 2.3]{CF2013} we know that $p_j(\cdot)\in \mathrm{LH}$ if and only if $1/p_j(\cdot)\in \mathrm{LH}.$ Hence, $\frac{1}{p(\cdot)}=\sum_{j=1}^m\frac{1}{p_j(\cdot)}$ is a linear combinations of functions in $ \mathrm{LH}$. Thus, combining this with the fact that $\mathrm{LH}$ is closed under linear combinations we conclude that $p(\cdot)\in \mathrm{LH}$.
\end{enumerate}
\end{definition}

For technical reasons and also for comparing with earlier definitions that were given in the literature, we observe the following.

\begin{lemma}\label{lem:range_exponents}
Let $(\vec{p}(\cdot), q(\cdot), \vec{r}, s)$ be a $m$-admissible quadruple. Set
    \begin{equation*}
    \rho_1 := r
    \end{equation*}
and define $\rho_2\in(0,\infty)$ via
    \begin{equation*}
    \frac{1}{\rho_1} - \frac{1}{\rho_2} := \gamma.
    \end{equation*}
Moreover, set
    \begin{equation*}
    s_2 := s
    \end{equation*}
and define $s_1\in(0,\infty]$ via
    \begin{equation*}
    \frac{1}{s_1} - \frac{1}{s_2} := \gamma.
    \end{equation*}
Then, there holds
    \begin{equation*}
    \rho_1 < p_{-} \leq p_{+} < s_1,
    \end{equation*}
and
    \begin{equation*}
    \rho_2 < q_{-} \leq q_{+} < s_2.
    \end{equation*}
Moreover, $\rho_1\leq\rho_2$ and $s_1\leq s_2$. In particular, if $\rho_1\geq1$, then also $\rho_2\geq1$.
\end{lemma}

\begin{proof}
The assertions of the last two sentences are obvious. Thus, we concentrate on proving the rest of the lemma.
    
First of all, observe that
\begin{equation*}
    \frac{1}{r} = \sum_{i=1}^{m}\frac{1}{r_j} > \sum_{i=1}^{m}\frac{1}{(p_j)_{-}} \geq \sum_{i=1}^{m}\frac{1}{p_j(\cdot)} = \frac{1}{p(\cdot)},
\end{equation*}
so by taking the essential supremum we arrive at
\begin{equation*}
    \rho_1 = r < p_{-}.
\end{equation*}
Moreover, we have
\begin{equation}\label{eq:pqg}
    \frac{1}{p(\cdot)} = \frac{1}{q(\cdot)} + \gamma,
\end{equation}
therefore by taking the essential supremum we arrive at
\begin{equation*}
    \frac{1}{p_{-}} = \frac{1}{q_{-}} + \gamma,
\end{equation*}
so in particular
\begin{equation*}
    \frac{1}{\rho_1} = \frac{1}{r} > \frac{1}{p_{-}} > \gamma,
\end{equation*}
showing in particular that indeed we can define $\rho_2\in(0,\infty)$ as in the statement of the lemma. Furthermore, we have
\begin{equation*}
    q_{-} = \frac{1}{\frac{1}{p_{-}} - \gamma} > \frac{1}{\frac{1}{\rho_1} - \gamma} = \rho_2.
\end{equation*}
In addition, we have $q_{+} < s = s_2$ by definition. Since
\begin{equation*}
    \frac{1}{s_2} + \gamma \geq0,
\end{equation*}
one can indeed define $s_1\in(0,\infty]$ as in the statement of the lemma. Finally, by taking the essential infimum of both sides of \eqref{eq:pqg}, we have
\begin{equation*}
    \frac{1}{p_{+}} = \frac{1}{q_{+}} + \gamma.
\end{equation*}
Thus, we compute
\begin{equation*}
    p_{+} = \frac{1}{\gamma + \frac{1}{q_{+}}} < \frac{1}{\gamma + \frac{1}{s}} = s_1,
\end{equation*}
concluding the proof.
\end{proof}

\begin{definition}
Let $(\vec{p}(\cdot),q(\cdot),\vec{r},s)$ be a $m$-admissible quadruple with
\begin{equation*}
    \frac{1}{p(\cdot)} - \frac{1}{q(\cdot)} = \gamma\in[0,\infty).
\end{equation*}
We say that a vector of weights $\vec{w}=(w_1,\dots,w_m)$ satisfies the $m$-linear $\mathcal{A}_{({\vec{p}(\cdot),q(\cdot)}),(\vec{r},s)}$ condition (or $\vec{w}\in \mathcal{A}_{({\vec{p}(\cdot),q(\cdot)}),(\vec{r},s)}$) if

\begin{equation}\label{eq:Apqrs constant}
    [\vec{w}]_{\mathcal{A}_{(\vec{p}(\cdot),q(\cdot)),(\vec{r},s)}}:=\sup_{Q}|Q|^{\gamma - \left(\frac{1}{r}-\frac{1}{s}\right)}
    \Vert \nu_{\vec{w}}\chi_{Q}\Vert_{\frac{1}{\frac{1}{q(\cdot)}-\frac{1}{s}}}\prod_{j=1}^{m}\Vert w_j^{-1}\chi_{Q}\Vert_{\frac{1}{\frac{1}{r_j}-\frac{1}{p_j(\cdot)}}}<\infty,
\end{equation}
where the supremum is taken over all cubes $Q\subset\R^n$ and $\nu_{\vec{w}} := \prod_{j=1}^{m}w_j$.

If $\gamma=0$, we denote
\begin{equation*}
    \mathcal{A}_{\vec{p}(\cdot),(\vec{r},s)} := \mathcal{A}_{({\vec{p}(\cdot),q(\cdot)}),(\vec{r},s)}.
\end{equation*}
Moreover, if $r_j=1$ and $s=\infty$, then we denote
\begin{equation*}
    \mathcal{A}_{\vec{p}(\cdot),q(\cdot)} := \mathcal{A}_{({\vec{p}(\cdot),q(\cdot)}),(\vec{r},s)}.
\end{equation*}
\end{definition}

\begin{remark}\label{rmk:compare_definitions}
\noindent\begin{enumerate}
    \item If $r_j=1$, $s=\infty$ and $\gamma\in[0,m)$, then $\mathcal{A}_{({\vec{p}(\cdot),q(\cdot)}),(\vec{r},s)}=\mathcal{A}_{\vec{p}(\cdot),q(\cdot)}$. This recovers the multilinear off-diagonal variable weight class $\mathcal{A}_{\vec{p}(\cdot),q(\cdot)}$ that was first introduced in \cite[Definition 1.3]{CHSW2025}. Moreover, in the special case that $\gamma=0$, then $\mathcal{A}_{\vec{p}(\cdot),q(\cdot)}=\mathcal{A}_{\vec{p}(\cdot)}$. The latter bi/multilinear diagonal variable weight class $\mathcal{A}_{\vec{p}(\cdot)}$ was first introduced in \cite{CG2020}.
    \item\label{itm:Aprs} If $m=1$ and $\gamma\in[0,1)$, then $\mathcal{A}_{({\vec{p}(\cdot),q(\cdot)}),(\vec{r},s)}$ by Lemma~\ref{lem:range_exponents} becomes $\mathcal{A}_{(p(\cdot),q(\cdot)),(\vec{\rho},\vec{s})}$ with $\vec{\rho}=(\rho_1,\rho_2)=\left(r,\frac{1}{-\gamma+\frac{1}{r}}\right)$ and $\vec{s}=(s_1,s_2)=\left(\frac{1}{\gamma+\frac{1}{s}},s\right)$, where the linear limited range variable weight class $\mathcal{A}_{(p(\cdot),q(\cdot)),(\vec{\rho},\vec{s})}$ was first introduced in \cite[page 6]{Nieraeth_2023} but the notation stems from the work \cite[Definition 1.4]{LO2025}. 
    \item By \eqref{itm:Aprs} if $m=1$ and $\gamma=0$, then this definition coincides  with the linear diagonal/limited range variable weight class $\mathcal{A}_{p(\cdot),(r,s)}$ first introduced in \cite[Definition 3.4]{Nieraeth_2023}. If in addition, we assume that $r_j=1$ and $s=\infty$, then $\mathcal{A}_{p(\cdot),(r,s)}=\mathcal{A}_{p(\cdot)}$, where the linear diagonal variable weight class $\mathcal{A}_{p(\cdot)}$ was first introduced in \cite[(1.3) in page 365]{CruzUribe2011}.          
\end{enumerate}
\end{remark}

\begin{definition}
Given variable exponents 
\begin{equation*}
p_1(\cdot),\dots,p_m(\cdot),\qquad p_{0,1}(\cdot),\dots,p_{0,m}(\cdot)\qquad \text{and}\qquad p_{1,1}(\cdot),\dots,p_{1,m}(\cdot)
\end{equation*}
and vectors of weights 
\begin{equation*}
\vec{w}=(w_1,\dots,w_m),\qquad\vec{w}_0=(w_{0,1},\dots,w_{0,m})\qquad\text{and}\qquad\vec{w}_1=(w_{1,1},\dots,w_{1,m})
\end{equation*}
we define 
\begin{gather*}
    \vec{p}(\cdot):=(p_1(\cdot),\dots,p_m(\cdot)),\quad
    \vec{p}_0(\cdot):=(p_{0,1}(\cdot),\dots,p_{0,m}(\cdot))
    \quad\text{and}\quad
    \vec{p}_1(\cdot):=(p_{1,1}(\cdot),\dots,p_{1,m}(\cdot))\\
    \nu_{\vec{w}}:=\prod_{j=1}^m w_j,\qquad \nu_{\vec{w}_0}:=\prod_{j=1}^m w_{0,j}\qquad\text{and}\qquad \nu_{\vec{w}_1}:=\prod_{j=1}^m w_{1,j}.
\end{gather*}
\end{definition}

\section{Auxiliary results}\label{sec: auxiliary results}

In this section, we recall several known results from the theory of variable Lebesgue spaces that we will need for our purposes.  

\begin{lemma}[\cite{CF2013}, Proposition 2.18]\label{lem:homog}
Let $p(\cdot)\in\mathscr{P}_0$. Then, it holds that  $\||f|^s\|_{p(\cdot)}=\|f\|^s_{sp(\cdot)},$ for all constants $s\in(0,\infty)$. 
\end{lemma}

The following H\"older's inequalities hold in variable Lebesgue spaces.

\begin{lemma}[\cite{CF2013}, Corollary 2.30]\label{Holder's ineq.} 
Fix a positive integer $m\ge 2$ and let $p_1(\cdot),\dots,p_m(\cdot)\in\mathcal{P}$ satisfy
\begin{equation*}
    \sum_{j=1}^m\frac{1}{p_{j}(\cdot)}=1,
\end{equation*}
where $j=1,\dots,m$. Then, for all $f_j\in L^{p_j(\cdot)}$, one has
\begin{equation*}
   \int_{\R^n}|f_1\cdots f_m|\lesssim\|f_1\|_{L^{p_1(\cdot)}}\cdots\|f_m\|_{L^{p_m(\cdot)}},
\end{equation*}
where the implicit constant depends only on the $p_j(\cdot)$ and $m.$
\end{lemma}

\begin{lemma}[\cite{TWX2025}, Lemma 6.2]\label{Gen. Holder's ineq.}
Let $m\ge2$ and $p(\cdot),p_1(\cdot),\dots,p_m(\cdot)\in\mathscr{P}_{0}$ satisfy
\begin{equation*}
    \frac{1}{p(\cdot)}=\sum_{j=1}^m\frac{1}{p_j(\cdot)},
\end{equation*}
where $j=1,\dots,m$. Then, for all $f_j\in L^{p_j(\cdot)}$ and $f_1\cdots f_m\in L^{p(\cdot)}$ one has
\begin{equation}\label{Holder's impl. constant}
    \|f_1\cdots f_m\|_{L^{p(\cdot)}}\lesssim\|f_1\|_{L^{p_1(\cdot)}}\cdots\|f_m\|_{L^{p_m(\cdot)}},
\end{equation}
where the implicit constant depends only on the $p_j(\cdot)$ and $m.$
\end{lemma}

We point out that H\"older's inequality for two variable exponents can be found in \cite[Corollary 2.28]{CF2013} (see also \cite[Remark 2.6]{LO2025}). Moreover, the results \cite[Corollary 2.28]{CF2013} and \cite[Lemma 6.2]{TWX2025} were originally stated and proved for variable exponents that belong to the set $\mathcal{P}$. However, by tracking their proofs one can see that they also hold for variable exponents that belong to the set $\mathscr{P}_0$.

Now, we emphasize that it is possible to characterize the multilinear $\mathcal{A}_{({\vec{p}(\cdot),q(\cdot)}),(\vec{r},s)}$ weights in terms of the $\mathcal{A}_{(p(\cdot),q(\cdot)),(r,s)}$ and $\mathcal{A}_{p(\cdot),(r,s)}$ weight classes. In particular, relying on similar arguments as in \cite[Lemma 1.2]{CHSW2025}, \cite[Proposition 3.1.6]{Nieraeth2021} or \cite[Proposition 4.7]{CG2020} we obtain the following result.

\begin{lemma}\label{lem:Apqrs char.}
Let $(\vec{p}(\cdot),q(\cdot),\vec{r},s)$ be a $m$-admissible quadruple. Then $\vec{w}\in\mathcal{A}_{({\vec{p}(\cdot),q(\cdot)}),(\vec{r},s)}$ if and only if    
\begin{equation}\label{equiv. mult. weights 1}
\begin{cases}
    w_j\in\mathcal{A}_{p_j(\cdot),(r_j,\sigma_j)}, &\text{with}\quad\frac{1}{\sigma_j}=\frac{1}{r_j}-\big(\frac{1}{r}-\frac{1}{s}\big)\quad\text{for all}\quad j=1,\dots,m, \\
  \nu_{\vec{w}}\in\mathcal{A}_{(p(\cdot),q(\cdot)),(r,s)}.
\end{cases}
\end{equation}
\end{lemma}

\begin{proof}
Fix $j_0\in\{1,\dots,m\}$. Notice that $(p_{j_0}(\cdot),p_{j_0}(\cdot),r_{j_0},\sigma_{j_0})$ is a $1$-admissible quadruple. This follows from the fact that $(p_{j_0})_{+} < \sigma_{j_0}$. Indeed, this inequality is equivalent to $\frac{1}{\sigma_{j_0}}<\frac{1}{(p_{j_0})_{+}}$. Then, this is further equivalent to
\begin{equation*}
    \frac{1}{r_{j_0}}-\frac{1}{r}+\frac{1}{s}<\frac{1}{(p_{j_0})_{+}}.    
\end{equation*}
By rearrangement, this becomes
\begin{equation*}
    \frac{1}{r_{j_0}}-\frac{1}{(p_{j_0})_{+}}<\frac{1}{r}-\frac{1}{s}.
\end{equation*}
We estimate
\begin{equation*}
    \frac{1}{r} - \frac{1}{p(\cdot)} = \frac{1}{r_{j_0}}-\frac{1}{p_{j_{0}}(\cdot)}+\sum_{\substack{j=1\\j\neq j_0}}^m\frac{1}{r_j}-\frac{1}{p_j(\cdot)}\geq \frac{1}{r_{j_0}}-\frac{1}{p_{j_{0}}(\cdot)}.
\end{equation*}
Hence, taking essential supremum, we obtain
\begin{equation*}
    \frac{1}{r}-\frac{1}{s} \overset{\text{Lemma~\ref{lem:range_exponents}}}{>}\frac{1}{r}-\frac{1}{p_{+}}\geq\frac{1}{r_{j_0}}-\frac{1}{(p_{j_{0}})_{+}},
\end{equation*}
as wanted.

Similarly, by Lemma~\ref{lem:range_exponents} we have that $(p(\cdot),q(\cdot),r,s)$ is a $1$-admissible quadruple.

Now, we first prove the ``only if'' direction, that is, $\vec{w}\in\mathcal{A}_{({\vec{p}(\cdot),q(\cdot)}),(\vec{r},s)}\Rightarrow$ \eqref{equiv. mult. weights 1}. Notice that
\begin{equation*}
    \frac{1}{p_{j_0}(\cdot)}-\frac{1}{\sigma_{j_0}}=\frac{1}{p(\cdot)}-\frac{1}{s}+\sum_{\substack{j=1\\j\neq j_0}}^{m}\frac{1}{r_j}-\frac{1}{p_j(\cdot)}\qquad\text{and}\qquad\frac{1}{p(\cdot)}=\frac{1}{q(\cdot)}+\gamma.  
\end{equation*}
Hence, by H\"older's inequality (Lemma \ref{Gen. Holder's ineq.}) with variable exponents, we have
\begin{align*}
    \big\|w_{j_0}\chi_Q\big\|_{\frac{1}{\frac{1}{p_{j_0}(\cdot)}-\frac{1}{\sigma_{j_0}}}}\big\|w_{j_0}^{-1}\chi_Q\big\|_{\frac{1}{\frac{1}{r_{j_0}}-\frac{1}{p_{j_0}(\cdot)}}}
    &=\Bigg\|\nu_{\vec{w}}\prod_{\substack{j=1\\j\neq j_0}}^{m}w_{j_0}^{-1}\chi_Q\Bigg\|_{\frac{1}{\frac{1}{p(\cdot)}-\frac{1}{s}+\sum_{\substack{j=1\\j\neq j_0}}^{m}\frac{1}{r_j}-\frac{1}{p_j(\cdot)}}}\Big\|w_{j_0}^{-1}\chi_Q\Big\|_{\frac{1}{\frac{1}{r_{j_0}}-\frac{1}{p_{j_0}(\cdot)}}}\\ 
    &\lesssim\big\|\nu_{\vec{w}}\chi_Q\big\|_{\frac{1}{\frac{1}{p(\cdot)}-\frac{1}{s}}}\prod_{j=1}^m\big\|w_j^{-1}\chi_Q\big\|_{\frac{1}{\frac{1}{r_j}-\frac{1}{p_j(\cdot)}}}\\
    &\lesssim|Q|^\gamma\big\|\nu_{\vec{w}}\chi_Q\big\|_{\frac{1}{\frac{1}{q(\cdot)}-\frac{1}{s}}}\prod_{j=1}^m\big\|w_j^{-1}\chi_Q\big\|_{\frac{1}{\frac{1}{r_j}-\frac{1}{p_j(\cdot)}}}.
\end{align*}
Multiplying both sides of the previous estimate by $|Q|^{-(\frac{1}{r}-\frac{1}{s})}$ and taking a supremum over all cubes $Q$ we get that
\begin{equation*}
    [w_{j_0}]_{\mathcal{A}_{p_{j_0}(\cdot),(r,s)}}\lesssim[\vec{w}]_{\mathcal{A}_{(\vec{p}(\cdot),q(\cdot)),(\vec{r},s)}}.
\end{equation*}
For the claim that $\nu_{\vec{w}}\in\mathcal{A}_{(p(\cdot),q(\cdot)),(r,s)}$, first notice that
\begin{equation*}
    \frac{1}{r}-\frac{1}{p(\cdot)}=\sum_{j=1}^m\frac{1}{r_j}-\frac{1}{p_j(\cdot)}.      
\end{equation*}
Thus, by H\"older's inequality (Lemma \ref{Gen. Holder's ineq.}) with variable exponents, we obtain
\begin{equation*}
    \big\|\nu_{\vec{w}}\chi_Q\big\|_{\frac{1}{\frac{1}{q(\cdot)}-\frac{1}{s}}}\big\|\nu_{\vec{w}}^{-1}\chi_Q\big\|_{\frac{1}{\frac{1}{r}-\frac{1}{p(\cdot)}}}\lesssim\big\|\nu_{\vec{w}}\chi_Q\big\|_{\frac{1}{\frac{1}{q(\cdot)}-\frac{1}{s}}}\prod_{j=1}^m\big\|w_j^{-1}\chi_Q\big\|_{\frac{1}{\frac{1}{r_j}-\frac{1}{p_j(\cdot)}}}.   
\end{equation*}
 Multiplying both sides of this estimate by $|Q|^{\gamma-(\frac{1}{r}-\frac{1}{s})}$ and taking a supremum over all cubes $Q$ we deduce that
 \begin{equation*}
    [\nu_{\vec{w}}]_{\mathcal{A}_{(p(\cdot),q(\cdot)),(r,s)}}\lesssim[\vec{w}]_{\mathcal{A}_{(\vec{p}(\cdot),q(\cdot)),(\vec{r},s)}}.  
 \end{equation*}

To prove the ``if'' direction, we assume that \eqref{equiv. mult. weights 1} holds. Then, by Lemma \ref{lem:homog} and H\"older's inequality (Lemma \ref{Holder's ineq.}) with variable exponents 
\begin{equation*}
    \frac{\frac{1}{r}-\frac{1}{p(\cdot)}}{m(\frac{1}{r}-\frac{1}{s})}+\sum_{j=1}^m\frac{\frac{1}{p_j(\cdot)}-\frac{1}{\sigma_j}}{m(\frac{1}{r}-\frac{1}{s})}=1
\end{equation*}
we get 
\begin{align*}
    1&=\bigg(\displaystyle\stackinset{c}{}{c}{}{-\mkern4mu}{\displaystyle\int_Q}(\nu_{\vec{w}}^{-1})^{\frac{1}{m(\frac{1}{r}-\frac{1}{s})}}(w_1^{\frac{1}{m(\frac{1}{r}-\frac{1}{s})}}\cdots w_m^{\frac{1}{m(\frac{1}{r}-\frac{1}{s})}})\bigg)^{m(\frac{1}{r}-\frac{1}{s})}\\
    &\lesssim|Q|^{-m(\frac{1}{r}-\frac{1}{s})}\Bigg\|(\nu_{\vec{w}}^{-1})^{\frac{1}{m(\frac{1}{r}-\frac{1}{s})}}\chi_Q\Bigg\|_{\frac{m(\frac{1}{r}-\frac{1}{s})}{\frac{1}{r}-\frac{1}{p(\cdot)}}}^{m(\frac{1}{r}-\frac{1}{s})}\prod_{j=1}^m\Bigg\|w_j^{\frac{1}{m(\frac{1}{r}-\frac{1}{s})}}\chi_Q\Bigg\|_{\frac{m(\frac{1}{r}-\frac{1}{s})}{\frac{1}{p_j(\cdot)}-\frac{1}{\sigma_j}}}^{m(\frac{1}{r}-\frac{1}{s})}\\
    &=|Q|^{-m(\frac{1}{r}-\frac{1}{s})}\|\nu_{\vec{w}}^{-1}\chi_Q\|_{\frac{1}{\frac{1}{r}-\frac{1}{p(\cdot)}}}\prod_{j=1}^m\|w_j\chi_Q\|_{\frac{1}{\frac{1}{p_j(\cdot)}-\frac{1}{\sigma_j}}}.
\end{align*}
Hence, this implies that
\begin{align*}
    \|\nu_{\vec{w}}\chi_Q\|_{\frac{1}{\frac{1}{q(\cdot)}-\frac{1}{s}}}\prod_{j=1}^m\|w_j^{-1}\chi_Q\|_{\frac{1}{\frac{1}{r_j}-\frac{1}{p_j(\cdot)}}}
    &\lesssim|Q|^{-m(\frac{1}{r}-\frac{1}{s})}\|\nu_{\vec{w}}\chi_Q\|_{\frac{1}{\frac{1}{q(\cdot)}-\frac{1}{s}}}\|\nu_{\vec{w}}^{-1}\chi_Q\|_{\frac{1}{\frac{1}{r}-\frac{1}{p(\cdot)}}}\\
    &\qquad\times\prod_{j=1}^m\|w_j\chi_Q\|_{\frac{1}{\frac{1}{p_j(\cdot)}-\frac{1}{\sigma_j}}}\|w_j^{-1}\chi_Q\|_{\frac{1}{\frac{1}{r_j}-\frac{1}{p_j(\cdot)}}}\\
    &\lesssim|Q|^{-m(\frac{1}{r}-\frac{1}{s})}|Q|^{-\gamma+\frac{1}{r}-\frac{1}{s}}|Q|^{m(\frac{1}{r}-\frac{1}{s})}.
\end{align*}
Thus,
\begin{equation*}
    |Q|^{\gamma-(\frac{1}{r}-\frac{1}{s})}\|\nu_{\vec{w}}\chi_Q\|_{\frac{1}{\frac{1}{q(\cdot)}-\frac{1}{s}}}\prod_{j=1}^m\|w_j^{-1}\chi_Q\|_{\frac{1}{\frac{1}{r_j}-\frac{1}{p_j(\cdot)}}}\lesssim 1.
\end{equation*}
Taking a supremum over all cubes $Q$ we conclude that $\vec{w}\in\mathcal{A}_{({\vec{p}(\cdot),q(\cdot)}),(\vec{r},s)}$, which completes the proof. 
\end{proof}

\begin{remark}
In view of the following Lemma \ref{lem:rescale}, notice that Lemma \ref{lem:Apqrs char.} extends \cite[Lemma 1.2]{CHSW2025} and \cite[Proposition 4.7]{CG2020}. In particular, our result deals with variable weight classes that contain the linear/multilinear full range diagonal/off-diagonal weight classes of \cite[Lemma 1.2]{CHSW2025}.
\end{remark}

\begin{lemma}\label{lem:rescale}
Let $(\vec{p}(\cdot),\vec{q}(\cdot),\vec{r},s)$ be a $m$-admissible quadruple. Let also $a\in(0,\infty)$. Then, we have
\begin{equation*}
    \vec{w} \in \mathcal{A}_{(\vec{p}(\cdot),q(\cdot)),(\vec{r},s)} \Leftrightarrow 
    \vec{w}^{\frac{1}{a}} \in \mathcal{A}_{(a\vec{p}(\cdot),aq(\cdot)),(a\vec{r},as)}.
\end{equation*}
Moreover, in this case there holds
\begin{equation*}
    [\vec{w}^{\frac{1}{a}}]_{\mathcal{A}_{(a\vec{p}(\cdot),aq(\cdot)),(a\vec{r},as)}} = [\vec{w}]_{\mathcal{A}_{(\vec{p}(\cdot),q(\cdot)),(\vec{r},s)}}^{\frac{1}{a}}.
\end{equation*}
\end{lemma}

\begin{proof}
The proof follows immediately by recalling \eqref{eq:Apqrs constant} and using Lemma \ref{lem:homog}. Indeed, we have
\begin{align*}
    &|Q|^{\frac{\gamma}{a} - \left(\frac{1}{ar}-\frac{1}{as}\right)}
    \Vert \nu_{\vec{w}}^{\frac{1}{a}}\chi_{Q}\Vert_{\frac{1}{\frac{1}{aq(\cdot)}-\frac{1}{as}}}\prod_{j=1}^{m}\Vert w_j^{-\frac{1}{a}}\chi_{Q}\Vert_{\frac{1}{\frac{1}{ar_j}-\frac{1}{ap_j(\cdot)}}}\\
    &\qquad=\left(|Q|^{\gamma - \left(\frac{1}{r}-\frac{1}{s}\right)}
    \Vert \nu_{\vec{w}}\chi_{Q}\Vert_{\frac{1}{\frac{1}{q(\cdot)}-\frac{1}{s}}}\prod_{j=1}^{m}\Vert w_j^{-1}\chi_{Q}\Vert_{\frac{1}{\frac{1}{r_j}-\frac{1}{p_j(\cdot)}}}\right)^{\frac{1}{a}}.
\end{align*}
Taking a supremum over all cubes $Q$ we obtain the desired conclusion.
\end{proof}

\section{Factorization results for multilinear Muckenhoupt weights with variable exponents}\label{sec: factorizations}

One of the key ingredients for the proof of Theorem \ref{thm:main_result} is the following factorization result for the variable multilinear Muckenhoupt weights. The proof of this result is based on similar ideas of \cite{HL2023, HL2022, LO2025}.

\begin{lemma}\label{lem:KeyLemma}
Let $(\vec{p}(\cdot),q(\cdot),\vec{r},s)$ and $(\vec{p}_1(\cdot),q_1(\cdot),\vec{r},s)$ be proper $m$-admissible quadruples such that
\begin{equation*}
    \frac{1}{p_1(\cdot)}-\frac{1}{q_1(\cdot)}=\frac{1}{p(\cdot)}-\frac{1}{q(\cdot)}=\gamma\in[0,\infty).
\end{equation*}
Moreover, let
\begin{equation*}
    \vec{w}\in\mathcal{A}_{(\vec{p}(\cdot),q(\cdot)),(\vec{r},s)}
\end{equation*}
and
\begin{equation*}
    \vec{w}_1\in\mathcal{A}_{(\vec{p}_1(\cdot),q_1(\cdot)),(\vec{r},s)}.
\end{equation*}
Given any $\theta\in(0,1)$, define $\vec{p}_0(\cdot)\in(\mathscr{P}_0)^m$, $q_0(\cdot)\in\mathscr{P}_0$ by
\begin{align*}
    \frac{1}{p_j(\cdot)}&=\frac{1-\theta}{p_{0,j}(\cdot)}+\frac{\theta}{p_{1,j}(\cdot)},\quad j=1,\ldots,m,\\
    \frac{1}{q(\cdot)}&=\frac{1-\theta}{q_0(\cdot)}+\frac{\theta}{q_1(\cdot)},
\end{align*}
as well as the vector of weights $\vec{w}_0$ by
\begin{align*}
    w_j&=w_{0,j}^{1-\theta}w_{1,j}^{\theta},\quad j=1,\ldots,m.
\end{align*}
Then, there exists $\eta>0$ such that for all $\theta\in(0,\eta)$ all of the following hold:
\begin{enumerate}
    \item $(\vec{p}_0(\cdot), q_0(\cdot), \vec{r}, s)$ is a proper $m$-admissible quadruple with 
\begin{equation*}
    \quad\frac{1}{p_0(\cdot)}-\frac{1}{q_0(\cdot)}=\gamma.
\end{equation*}

    \item There holds
\begin{equation*}
    \vec{w}_0\in\mathcal{A}_{(\vec{p}_0(\cdot),q_0(\cdot)),(\vec{r},s)}.
\end{equation*}

    \item If $(q_1)_{-},q_{-}>1$, then also $(q_{0})_{-}>1$.
\end{enumerate}
\end{lemma}

The proof relies on \cite[Lemma 3.20]{LO2025}. Since the notation used there is slightly different from the one used here (see \eqref{itm:Aprs} of Remark~\ref{rmk:compare_definitions} for a comparison), we restate \cite[Lemma 3.20]{LO2025} for the reader's convenience in our notation.

\begin{lemma}[{\cite[Lemma 3.20]{LO2025}}]\label{lem:factor_linear}
Let $(p(\cdot),q(\cdot),r,s)$ and $(p_1(\cdot),q_1(\cdot),r,s)$ be proper $1$-admissible quadruples such that $r\geq1$ and
\begin{equation*}
    \frac{1}{p_1(\cdot)}-\frac{1}{q_1(\cdot)}=\frac{1}{p(\cdot)}-\frac{1}{q(\cdot)}=\gamma\in[0,1).
\end{equation*}
Let also
\begin{equation*}
    w\in\mathcal{A}_{(p(\cdot),q(\cdot)),(r,s)}
\end{equation*}
and
\begin{equation*}
    \vec{w}_1\in\mathcal{A}_{(p_1(\cdot),q_1(\cdot)),(r,s)}.
\end{equation*}
Given any $\theta\in(0,1)$, define $p_0(\cdot)\in\mathscr{P}_0$, $q_0(\cdot)\in\mathscr{P}_0$ by
\begin{align*}
    \frac{1}{p(\cdot)}&=\frac{1-\theta}{p_{0}(\cdot)}+\frac{\theta}{p_{1}(\cdot)},\\
    \frac{1}{q(\cdot)}&=\frac{1-\theta}{q_0(\cdot)}+\frac{\theta}{q_1(\cdot)},
\end{align*}
as well as the weight $w_0$ by
\begin{align*}
    w&=w_{0}^{1-\theta}w_{1}^{\theta}.
\end{align*}
Then, there exists $\eta>0$, such that for all $\theta\in(0,\eta)$ all of the following hold:
\begin{enumerate}
    \item $(p_0(\cdot),q_0(\cdot),r,s)$ is a proper $1$-admissible quadruple with $\frac{1}{p_0(\cdot)} - \frac{1}{q_0(\cdot)} = \gamma$.

    \item There holds
\begin{equation*}
    w_0\in\mathcal{A}_{(p_0(\cdot),q_0(\cdot)),(r,s)}.
\end{equation*}
\end{enumerate}
\end{lemma}

It might seem at a first glance that in the way we have rephrased \cite[Lemma 3.20]{LO2025} and we have missed assumptions on the ranges of the exponents stipulated in \cite[Lemma 3.20]{LO2025}. However, as explained in \eqref{itm:Aprs} of Remark~\ref{rmk:compare_definitions}, this is not the case.

Now we prove Lemma~\ref{lem:KeyLemma}.

\begin{proof}[Proof of Lemma~\ref{lem:KeyLemma}]
Notice that once $\theta$ is fixed, $\vec{p}_0(\cdot),q_0(\cdot)$ and $\vec{w}_0(\cdot)$ are uniquely determined. It is also trivial to verify that
\begin{equation*}
    \frac{1}{p_0(\cdot)}-\frac{1}{q_0(\cdot)}=\gamma.
\end{equation*}

\medskip

\textbf{Step 1.} Fix $j\in\{1,\ldots,m\}$. Pick $a\in(0,\infty)$ with $ar_j\geq1$. Then, the quadruples
\begin{equation*}
    (ap_j(\cdot),ap_j(\cdot),ar_{j},m\sigma_j)\quad\text{and}\quad(ap_{1,j}(\cdot),ap_{1,j}(\cdot),ar_{j},a\sigma_{j})
\end{equation*}
are proper $1$-admissible with $ar_j\geq1$ and
\begin{equation*}
    \frac{1}{ap_j(\cdot)} - \frac{1}{ap_j(\cdot)} = \frac{1}{ap_{1,j}(\cdot)} - \frac{1}{ap_{1,j}(\cdot)} = 0 \in [0,1).
\end{equation*}
Moreover, in view of our assumptions, by combining Lemma~\ref{lem:Apqrs char.} with Lemma~\ref{lem:rescale}, we deduce that
\begin{equation*}
    w_{j}^{\frac{1}{a}}\in\mathcal{A}_{ap_{j}(\cdot),(ar_j,a\sigma_j)}
\end{equation*}
and
\begin{equation*}
    w_{1,j}^{\frac{1}{a}}\in\mathcal{A}_{ap_{1,j}(\cdot),(ar_j,a\sigma_j)}.
\end{equation*}
Clearly
\begin{equation*}
    w_{j}^{\frac{1}{a}} = (w_{0,j}^{\frac{1}{a}})^{1-\theta}(w_{1,j}^{\frac{1}{a}})^{\theta}.
\end{equation*}
Thus, by applying Lemma~\ref{lem:factor_linear}, we deduce that there exists $\eta_j>0$, so that for all $\theta\in(0,\eta_j)$, $(ap_{0,j}(\cdot),ap_{0,j}(\cdot),ar_{j},a\sigma_{j})$ is a proper 1-admissible quadruple and $w_{0}^{\frac{1}{a}} \in \mathcal{A}_{ap_{0,j}(\cdot),(ar_j,a\sigma_j)}$. By Lemma~\ref{lem:rescale}, this means $w_{0} \in \mathcal{A}_{p_{0,j}(\cdot),(r_j,\sigma_j)}$. Moreover, we have $ap_{0,j}(\cdot)\in  \mathrm{LH}$, therefore also $p_{0,j}(\cdot)\in\mathrm{LH}$.

\medskip

\textbf{Step 2.} Pick $b\in(0,\infty)$ with
\begin{equation*}
    b > \max\left\{\gamma, \frac{1}{r}\right\}.
\end{equation*}
Then, the quadruples
\begin{equation*}
    (bp(\cdot),bq(\cdot),br,bs)\quad\text{and}\quad(bp_{1}(\cdot),bq_{1}(\cdot),br,bs)
\end{equation*}
are proper $1$-admissible with $br\geq1$ and
\begin{equation*}
    \frac{1}{bp(\cdot)} - \frac{1}{bq(\cdot)} = \frac{1}{bp_{1}(\cdot)} - \frac{1}{bq_{1}(\cdot)} = \frac{\gamma}{b} \in [0,1).
\end{equation*}
Moreover, in view of our assumptions, by combining Lemma~\ref{lem:Apqrs char.} with Lemma~\ref{lem:rescale}, we deduce that
\begin{equation*}
    \nu_{\vec{w}}^{\frac{1}{b}}\in\mathcal{A}_{(bp(\cdot),bq(\cdot)),(br,bs)}
\end{equation*}
and
\begin{equation*}
    \nu_{\vec{w}_1}^{\frac{1}{b}}\in\mathcal{A}_{(bp_1(\cdot),bq_1(\cdot)),(br,bs)}.
\end{equation*}
Clearly
\begin{equation*}
    \nu_{\vec{w}}^{\frac{1}{b}} = (\nu_{\vec{w}_0}^{\frac{1}{b}})^{1-\theta}(\nu_{\vec{w}_1}^{\frac{1}{b}})^{\theta}.
\end{equation*}
Thus, by applying Lemma~\ref{lem:factor_linear}, we deduce that there exists $\eta_{\nu}>0$, so that for all $\theta\in(0,\eta_{\nu})$, $(bp_0(\cdot),bq_0(\cdot),br,bs)$ is a proper 1-admissible quadruple and $\nu_{\vec{w}_0}^{\frac{1}{b}}\in\mathcal{A}_{(bp_0(\cdot),bq_0(\cdot)),(br,bs)}$. By Lemma~\ref{lem:rescale}, this means $\nu_{\vec{w}_0}\in\mathcal{A}_{(p_0(\cdot),q_0(\cdot)),(r,s)}$. Moreover, we have $bp_{0}(\cdot),bq_0(\cdot)\in  \mathrm{LH}$, therefore also $p_{0}(\cdot),q_0(\cdot)\in\mathrm{LH}$.

\medskip

\textbf{Step 3.} Take $\eta := \min\{\eta_1,\ldots,\eta_m,\eta_{\nu}\}$. Then, the quadruple $(\vec{p}_0(\cdot),q_0(\cdot),\vec{r},s)$ is proper $m$-admissible. Moreover, we have
\begin{equation*}
    w_{0} \in \mathcal{A}_{p_{0,j}(\cdot),(r_j,\sigma_j)},\quad\forall j = 1,\ldots,m
\end{equation*}
and
\begin{equation*}
    \nu_{\vec{w}_0}\in\mathcal{A}_{(p_0(\cdot),q_0(\cdot)),(r,s)}.
\end{equation*}
By Lemma~\ref{lem:Apqrs char.} we deduce that $\vec{w}_0\in\mathcal{A}_{(\vec{p}_0(\cdot),q_0(\cdot)),(\vec{r},s)}$.

\medskip

\textbf{Step 4.} If $(q_1)_{-},q_{-}>1$, then by the proof of \cite[Lemma 3.1]{LO2025} we have that by picking $\eta$ even smaller we can ensure in addition that $(q_{0})_{-}>1$.
\end{proof}

In view of later applications of our extrapolation results, we also need a varied version of a special case of Lemma~\ref{lem:KeyLemma}. Even though the changes are restricted essentially to the notational level, we include a precise statement and a short explanation for the sake of clarity.

\begin{lemma}\label{lem:KeyLemma_with_t}
Let $(\vec{p}(\cdot),q(\cdot),\vec{1},\infty)$ and $(\vec{p}_1(\cdot),q_1(\cdot),\vec{1},\infty)$ be proper $m$-admissible quadruples with
\begin{equation*}
    \frac{1}{p(\cdot)}-\frac{1}{q(\cdot)}=\frac{1}{p_1(\cdot)}-\frac{1}{q_1(\cdot)}=\gamma\in[0,\infty).
\end{equation*}
Let $t\in(0,\infty)$ be a fixed constant with
\begin{equation*}
    t<\min\{(p_j)_{-}:~j=1,\ldots,m\}
\end{equation*}
and
\begin{equation*}
    t<\min\{(p_{1,j})_{-}:~j=1,\ldots,m\}.
\end{equation*}
Let
\begin{equation*}
    \vec{w}\in\mathcal{A}_{\vec{p}(\cdot),q(\cdot)}
\end{equation*}
and
\begin{equation*}
    \vec{w}_1\in\mathcal{A}_{\vec{p}_1(\cdot),q_1(\cdot)}
\end{equation*}
such that in addition one has
\begin{equation*}
    \vec{w}^{t}\in\mathcal{A}_{\frac{\vec{p}(\cdot)}{t},\frac{q(\cdot)}{t}}
\end{equation*}
and
\begin{equation*}
    \vec{w}_1\in\mathcal{A}_{\frac{\vec{p}_1(\cdot)}{t},\frac{q_1(\cdot)}{t}}.
\end{equation*}
Given any $\theta\in(0,1)$, define $\vec{p}_0(\cdot)\in(\mathscr{P}_0)^m$, $q_0(\cdot)\in\mathscr{P}_0$ by
\begin{align*}
    \frac{1}{p_j(\cdot)}&=\frac{1-\theta}{p_{0,j}(\cdot)}+\frac{\theta}{p_{1,j}(\cdot)},\quad j=1,\ldots,m,\\
    \frac{1}{q(\cdot)}&=\frac{1-\theta}{q_0(\cdot)}+\frac{\theta}{q_1(\cdot)},
\end{align*}
as well as the vector of weights $\vec{w}_0$ by
\begin{align*}
    w_j&=w_{0,j}^{1-\theta}w_{1,j}^{\theta},\quad j=1,\ldots,m.
\end{align*}
Then, there exists $\eta>0$ such that for all $\theta\in(0,\eta)$ all of the following hold:
\begin{enumerate}
    \item $(\vec{p}_0(\cdot),q_0(\cdot),\vec{1},\infty)$ is a proper $m$-admissible quadruple with $\frac{1}{p_0(\cdot)}-\frac{1}{q_0(\cdot)} = \gamma$.\label{itm:no 1}

    \item $\vec{w}_0\in\mathcal{A}_{\vec{p}_0(\cdot),q_0(\cdot)}$.\label{itm:no 2}

    \item If $(q_1)_{-},q_{-}>1$, then also $(q_{0})_{-}>1$.\label{itm:no 3}

    \item  $t<\min\{(p_{0,j})_{-}:~j=1,\ldots,m\}$.

    \item $\vec{w}_0^{t}\in\mathcal{A}_{\frac{\vec{p}_0(\cdot)}{t},\frac{q_0(\cdot)}{t}}$.
\end{enumerate}
\end{lemma}

\begin{proof}
By a direct application of Lemma~\ref{lem:KeyLemma} we have that items \eqref{itm:no 1}, \eqref{itm:no 2} and \eqref{itm:no 3} are all satisfied if we choose $\eta$ small enough.

Set now $\delta:=t\gamma\in[0,\infty)$ and define the ``$\pi\text{-}\kappa\text{-}\delta\text{-}\omega$'' counterparts
\begin{gather*}
    \vec{\omega}:=\vec{w}^t,\\
    \vec{\pi}(\cdot):=\frac{\vec{p}(\cdot)}{t},\\
    \kappa(\cdot):=\frac{q(\cdot)}{t},
\end{gather*}
and extend these pieces of notation analogously to their counterparts for the endpoint 1. Then, we clearly have that
\begin{equation*}
    (\vec{\pi}(\cdot),\kappa(\cdot),\vec{1},\infty)\quad\text{and}\quad(\vec{\pi}_1(\cdot),\kappa_1(\cdot),\vec{1},\infty)
\end{equation*}
are proper $m$-admissible quadruples with
\begin{equation*}
    \frac{1}{\pi_1(\cdot)}-\frac{1}{\kappa_1(\cdot)}=\frac{1}{\
    \pi(\cdot)}-\frac{1}{\kappa(\cdot)}=\delta.
\end{equation*}
Moreover, by assumption it holds that
\begin{gather*}
    \vec{\omega}_1\in\mathcal{A}_{\vec{\pi}_1(\cdot),\kappa_1(\cdot)},\\
    \vec{\omega}\in\mathcal{A}_{\vec{\pi}(\cdot),\kappa(\cdot)}.
\end{gather*}   
The ``$\pi\text{-}\kappa\text{-}\delta\text{-}\omega$'' counterparts for the endpoint 0 are defined directly in terms of the already defined ``$\pi\text{-}\kappa\text{-}\delta\text{-}\omega$'' counterparts by
\begin{align*}
    \frac{1}{\pi_j(\cdot)}&=\frac{1-\theta}{\pi_{0,j}(\cdot)}+\frac{\theta}{\pi_{1,j}(\cdot)},\quad j=1,\ldots,m,\\
    \frac{1}{\kappa(\cdot)}&=\frac{1-\theta}{\kappa_0(\cdot)}+\frac{\theta}{\kappa_1(\cdot)},
\end{align*}
and
\begin{align*}
    \omega_j&=\omega_{0,j}^{1-\theta}\omega_{1,j}^{\theta},\quad j=1,\ldots,m.
\end{align*}
Thus, by applying Lemma~\ref{lem:KeyLemma} for the ``$\pi\text{-}\kappa\text{-}\delta\text{-}\omega$'' counterparts, we can achieve that by picking $\eta$ small enough, the relation $\theta<\eta$ ensures that
\begin{itemize}
    \item $(\vec{\pi}_0(\cdot),\kappa_0(\cdot),\vec{1},\infty)$ is a proper $m$-admissible quadruple with $\frac{1}{\pi_0(\cdot)}-\frac{1}{\kappa_0(\cdot)}=\delta$, and

    \item $\vec{\omega}_0 \in \mathcal{A}_{\vec{\pi}_0(\cdot),\kappa_0(\cdot)}$.
\end{itemize}    
Therefore, it only remains to check that our definition of the ``$\pi\text{-}\kappa\text{-}\delta\text{-}\omega$'' counterparts for the endpoint 0 automatically yields the desired $t$-scaling, that is that we indeed have
\begin{equation}\label{eq:first_check}
    \vec{\pi}_0(\cdot)=\frac{\vec{p}_0(\cdot)}{t},\quad\kappa_0(\cdot)=\frac{q(\cdot)}{t},\quad\vec{\omega}_0=\vec{w}_0^{t}.
\end{equation}
Moreover, we need to check that we automatically have
\begin{equation}\label{eq:second_check}
    t<\min\{(p_{0,j})_{-}:~j=1,\ldots,m\}.
\end{equation}
To check~\eqref{eq:first_check}, observe that for all $j=1,\ldots,m$ we have
\begin{equation*}
    \frac{1}{\pi_{0,j}(\cdot)}=\frac{1}{1-\theta}\left(\frac{1}{\pi_j(\cdot)}-\frac{\theta}{\pi_{1,j}(\cdot)}\right)=\frac{1}{1-\theta}\left(\frac{t}{p_j(\cdot)}-\frac{t\theta}{p_{1,j}(\cdot)}\right)=
    \frac{t}{p_{0,j}(\cdot)}.
\end{equation*}
The other parts of~\eqref{eq:first_check} are similarly verified. Lastly, for checking~\eqref{eq:second_check}, we have
\begin{equation*}
    (\pi_{0,j})_{-}>1,\quad\forall j=1,\ldots,m
\end{equation*}
therefore
\begin{equation*}
    \left(\frac{p_{0,j}}{t}\right)_{-}>1,\quad\forall j=1,\ldots,m,
\end{equation*}
in other words
\begin{equation*}
    (p_{0,j})_{-}>t,\quad\forall j=1,\ldots,m,
\end{equation*}
concluding the proof.
\end{proof}

\section{Lions--Peetre compactness theorem in the bilinear setting}\label{sec: Lions--Peetre comp. thm.}

In this section we review the basic abstract tool about interpolation for bilinear compact operators that will be used in the proof of Theorem~\ref{thm:main_result}. For the reader's convenience, we include some background. None of the material here is new, but we carefully track the assumptions of the various definitions.

\subsection{Banach function spaces}

We begin with briefly recalling the notion of Banach function spaces. We follow the presentation and definitions from the recent paper \cite{LoNie2024} of Lorist--Nieraeth, where we refer to for more details, history and proofs.

\begin{definition}
Let $(\Omega,\mu)$ be a $\sigma$-finite measure space. We denote by $L^{0}(\Omega)$ the space of all measurable functions on $\Omega$. Let $X$ be a Banach space whose underlying set is a subset of $L^{0}(\Omega)$. Following \cite{LoNie2024}, the space $X$ is said to be a \emph{Banach function space over $(\Omega,\mu)$} if it satisfies both of the following two properties:
\begin{itemize}
    \item \textbf{Ideal property:} If $f\in X$ and $g\in L^0(\Omega)$ with $|g|\leq |f|$ $\mu$-a.e., then $g\in X$ and $\Vert g\Vert_{X}\leq\Vert f\Vert_{X}$.

    \item \textbf{Saturation property:} For every measurable set $E\subseteq\Omega$ with $\mu(E)>0$, there exists a measurable set $F\subseteq E$ with $\mu(F)>0$ and $\chi_{F}\in X$.
\end{itemize}
\end{definition}

\begin{remark}
In several places of the literature, the definition of Banach function spaces includes a much stronger property than the saturation property, namely, that the characteristic function of any measurable set of finite measure belongs to the space. However, as explained in \cite{LoNie2024}, this property is too restrictive, as it excludes in general for example weighted Lebesgue spaces, even with constant exponent.
\end{remark}

Several important Banach functions spaces enjoy in addition the Fatou property. While sometimes in the literature it furnished part of the definition of Banach functions spaces, for us, following \cite{LoNie2024}, it will be an additional, desirable property. We follow the formulation in \cite{LoNie2024}.

\begin{definition}
Let $X$ be a Banach function space over a $\sigma$-finite measure space $(\Omega,\mu)$. We say that $X$ has the \emph{Fatou property}, if for any pointwise increasing sequence $(f_k)_{k=1}^{\infty}$ of nonnegative functions in $X$ with $f_k\to f$ pointwise $\mu$-a.e.~as $k\to\infty$ for some $f\in L^{0}(\Omega)$ and $\sup\limits_{k=1,2,\ldots}\Vert f_k\Vert_{X}<\infty$, it holds $f\in X$ and $\Vert f\Vert_{X}=\lim\limits_{k\to\infty}\Vert f_k\Vert_{X}$.
\end{definition}

Another property that will be important for us is that of the absolutely continuous norm. We follow again the formulation in \cite{LoNie2024}.

\begin{definition}
Let $X$ be a Banach function space over a $\sigma$-finite measure space $(\Omega,\mu)$. We say that $X$ has a \emph{absolutely continuous norm}, if for all $f\in X$ and for all decreasing sequences $(E_k)_{k=1}^{\infty}$ of measurable sets with $\mu\left(\displaystyle\bigcap_{k=1}^{\infty}E_k\right)=0$, there holds $\lim\limits_{k\to\infty}\Vert f\chi_{E_k}\Vert_{X}=0$.
\end{definition}

Let us now remark that the weighted variable Lebesgue spaces we consider in this paper are Banach functions spaces satisfying the additional properties.

\begin{lemma}\label{lm:Weighted_Variable_Lebesgue_BFS}
Let $p(\cdot)\in\mathscr{P}$ and let $w$ be a weight on $\R^n$. Then, $L^{p(\cdot)}(\R^n)$ is a Banach function space over $\R^n$ (equipped with Lebesgue measure) that is reflexive as a Banach space with Banach dual $L^{p'(\cdot)}(w^{-1})$, satisfies the Fatou property and has absolutely continuous norm.
\end{lemma}

\begin{proof}
First of all, it is well-known that $L^{p(\cdot)}(\R^n)$ is a Banach function space over $\R^n$ that satisfies the Fatou property, see for instance \cite[Lemma 3.2.8 (b)]{DHHR2011},  and has absolutely continuous norm, see for example \cite[p.~73]{CF2013}. Then, by \cite[Proposition 3.17]{LoNie2024} we immediately obtain that $L^{p(\cdot)}(w)$ is a Banach function space over $\Omega$ that has the Fatou property. Furthermore, by combining \cite[Proposition 3.12]{LoNie2024} with \cite[Proposition 3.17]{LoNie2024} we deduce that $L^{p(\cdot)}(w)$ has absolutely continuous norm. Lastly, it is well-known that the \emph{K\"{o}the dual} of $L^{p(\cdot)}(\R^n)$ (see \cite{LoNie2024} for the precise, general definition of the K\"{o}the dual) is
\begin{equation*}
    (L^{p(\cdot)}(\R^n))'=L^{p'(\cdot)}(\R^n),
\end{equation*}
see for instance \cite[p.~73]{CF2013}. Thus, by \cite[Proposition 3.17]{LoNie2024} we obtain that the K\"{o}the dual of $L^{p(\cdot)}(w)$ is
\begin{equation*}
    (L^{p(\cdot)}(w))'=L^{p'(\cdot)}(w^{-1}).
\end{equation*}
In particular, $(L^{p(\cdot)}(w))'$ has also absolutely continuous norm. Combining \cite[Proposition 3.12]{LoNie2024} with \cite[Corollary 3.16]{LoNie2024} we immediately deduce that $L^{p(\cdot)}(w)$ is reflexive. From \cite[Proposition 3.15 (a)]{LoNie2024} we finally obtain that the Banach dual of $L^{p(\cdot)}(w)$ is $L^{p'(\cdot)}(w^{-1})$.
\end{proof}

As proved in \cite[Proposition 2.5]{LoNie2024}, the saturation property has several equivalent reformulations. One of them is so important for us that we state it explicitly:

\begin{lemma}\label{lm:Saturation_Exhaustion}
Let $X$ be a Banach function space over a $\sigma$-finite measure space $(\Omega,\mu)$. Then, there exists an increasing sequence of measurable sets $(A_k)^{\infty}_{k=1}$ such that $\mu(A_k)<\infty$ and $\chi_{A_k}\in X$ for all $k=1,2,\ldots$, and $\Omega=\bigcup_{k=1}^{\infty}A_k$.
\end{lemma}

\begin{proof}
By \cite[Proposition 2.5 (ii)]{LoNie2024} we have that there exists an increasing sequence of measurable sets $(F_k)^{\infty}_{k=1}$ such that $\chi_{F_k}\in X$ for all $k=1,2,\ldots$, and $\Omega=\bigcup_{k=1}^{\infty}F_k$. Moreover, since $\Omega$ is $\sigma$-finite, we can write $\Omega=\bigcup_{k=1}^{\infty}E_k$, where $(E_k)^{\infty}_{k=1}$ is an increasing sequence of measurable sets with $\mu(E_k)<\infty$, for all $k=1,2,\ldots$. Set $A_k:=F_k\cap E_k$, $k=1,2,\ldots$. Then, $(A_k)^{\infty}_{k=1}$ is an increasing sequence of measurable sets with $\mu(A_k)<\infty$, for all $k=1,2,\ldots$ and $\Omega=\bigcup_{k=1}^{\infty}A_k$. Moreover, since $0\leq\chi_{A_k}\leq\chi_{F_k}$, the ideal property yields $\chi_{A_k}\in X$, for all $k=1,2,\ldots$.
\end{proof}

We now turn to an important embedding property of weighted variable Lebesgue spaces, whose underlying weight satisfies an additional assumption present in our context.

\begin{definition}
Let $(\Omega,\mu)$ be a $\sigma$-finite measure space. Following \cite{Kouts2021}, we say that a sequence $(f_k)^{\infty}_{k=1}$ of measurable functions on $\Omega$ \emph{converges locally in measure} to a measurable function $f$ on $\Omega$, if for any measurable set $E\subseteq\Omega$ with $\mu(E)<\infty$ we have $f_k|_{E}\to f|_{E}$ as $k\to\infty$ in measure.
\end{definition}

The local convergence in measure was characterized through a metric on the space of measurable functions in \cite[Proposition B.6]{Kouts2021}.

\begin{lemma}\label{lm:conv_local_meas}
Let $(\Omega,\mu)$ be a $\sigma$-finite measure space. Let $(A_k)^{\infty}_{k=1}$ be any sequence of pairwise disjoint measurable sets of finite positive measure covering $\Omega$. Consider the map $d:L^{0}(\Omega)\times L^0(\Omega)\to[0,\infty)$ given by
\begin{equation*}
    d(f,g):=\sum_{k=1}^{\infty}\frac{1}{2^k|A_k|}\int_{A_k}\frac{|f-g|}{1+|f-g|}\dd x.
\end{equation*}
By \cite[Proposition B.6]{Kouts2021} we have that $d$ is a well-defined metric on $L^{0}(\Omega)$ making it into a topological vector space. Let $(f_k)^{\infty}_{k=1}$ be a sequence of measurable functions on $\Omega$ and let $f\in L^{0}(\Omega)$. Then, the following are equivalent:
\begin{enumerate}
        \item[(a)] $f_k\overset{d}{\to}f$ as $k\to\infty$.

        \item[(b)] $f_k\to f$ locally in measure as $k\to\infty$.

        \item[(c)] For any $k=1,2,\ldots$ we have $f_m|_{A_k}\to f|_{A_k}$ as $m\to\infty$ in measure. 
\end{enumerate}
\end{lemma}

\begin{proof}
That (a) implies (b) is just direction $\Rightarrow$ in \cite[Proposition B.6 (iii)]{Kouts2021}. It is clear that (b) implies (c). Finally, the proof of direction $\Leftarrow$ in \cite[Proposition B.6 (iii)]{Kouts2021} shows that (c) implies (a), since in that proof one only uses that $f_m|_{A_k}\to f|_{A_k}$ as $m\to\infty$ in measure, for all $k=1,2,\ldots$.
\end{proof}

The following estimate will turn out to be crucial for the embedding of weighted variable Lebesgue spaces into the space of measurable functions equipped with the topology of local convergence in measure.

\begin{lemma}\label{lm:Lp_in_local}
Let $p(\cdot)\in\mathscr{P}$ and let $w$ be a weight on $\R^n$. Then, there exists a sequence $(A_k)^{\infty}_{k=1}$ of pairwise disjoint measurable sets of finite measure covering $\R^n$, such for all $k=1,2,\ldots$, there exists a constant $C_{k}\in(0,\infty)$ with
\begin{equation*}
    \int_{A_k}|f(x)|\dd x\leq C_{k}\Vert f\Vert_{L^{p(\cdot)}(w)},\quad\forall f\in L^{p(\cdot)}(w).
\end{equation*}
\end{lemma}

\begin{proof}
By Lemma~\ref{lm:Saturation_Exhaustion}, applied for $L^{p'(\cdot)}(w^{-1})$, there exists an increasing sequence of measurable sets $(B_k)_{k=1}^{\infty}$ such that $\mu(B_k)<\infty$ and $\chi_{B_k}\in L^{p'(\cdot)}(w^{-1})$ covering $\R^n$. Set $A_1:=B_1$ and
\begin{equation*}
    A_k:=B_k\setminus\left(\bigcup_{m=1}^{k-1}B_m\right),\quad k=2,3,\ldots.
\end{equation*}
Then, $A_1,A_2,\ldots$ are measurable, pairwise disjoint, have finite measure and cover $\R^n$. Furthermore, $\chi_{A_k}\in L^{p'(\cdot)}(w^{-1})$ for all $k=1,2,\ldots$ by the ideal property.
    
Fix now $k\geq 1$.  Let $f\in L^{p(\cdot)}(w)$ be arbitrary. By H\"{o}lder's inequality in the variable exponent setting, Lemma~\ref{Holder's ineq.}, for $p(\cdot)$ and $p'(\cdot)$ we have
\begin{align*}
    \int_{A_{k}}|f(x)|\dd x&=\int_{\R^n}|f(x)w(x)|\chi_{A_{k}}(x)w^{-1}(x)\dd x\leq C\Vert fw\Vert_{L^{p(\cdot)}(\R^n)}\Vert\chi_{A_{k}}w^{-1}\Vert_{L^{p'(\cdot)}(\R^n)}\\
    &=C_{k}\Vert f\Vert_{L^{p(\cdot)}(w)},
\end{align*}
where $C_{k}:=C\Vert\chi_{A_k}\Vert_{L^{p'(\cdot)}(w^{-1})}$ is finite.
\end{proof}

Now we can state the embedding.

\begin{proposition}\label{prop:embedding_variable_Lebesgue}
Let $p(\cdot)\in\mathscr{P}$ and let $w$ be a weight on $\R^n$. Then, the inclusion $L^{p(\cdot)}(w)\subseteq L^{0}(\R^n)$ is a continuous embedding.
\end{proposition}

\begin{proof}
This is proved identically to \cite[Proposition 3.3.1.3]{Kouts2021}, by using Lemma~\ref{lm:conv_local_meas} (c) in tandem with Lemma~\ref{lm:Lp_in_local}, instead of assumption \cite[Definition 3.3.1.1 (iv)]{Kouts2021}.
\end{proof}

\subsection{The Cobos--Fern\'{a}ndez-Cabrera--Mart\'{i}nez theorem}

\begin{definition}
Let $\bar{A}=(A_0,A_1)$ be a pair of Banach spaces. We say that $\bar{A}$ is a \emph{Banach couple}, if there is some Hausdorff topological vector space $Y$ such that both $A_0$ and $A_1$ are continuously embedded in $Y$.
\end{definition}

We now recall the result of Cobos--Fern\'{a}ndez-Cabrera--Mart\'{i}nez \cite{CFCM2020} about complex interpolation for bilinear compact operators. To state it, we need two additional pieces of notations: if $\bar{A}=(A_0,A_1)$ is a Banach couple, the we denote by $A_j^{\circ}$ the norm topological closure of $A_0\cap A_1$ in $A_j$, for all $j=0,1$. Moreover, we denote by ${\bf\mathcal{B}}(\bar{A}\times\bar{B},\bar{E})$ the family of bilinear operators $T$ defined on $(A_0\cap A_1)\times(B_0\cap B_1)$ with values in $E_0\cap E_1$ that satisfy the following
\begin{equation*}
  \|T(a,b)\|_{E_j}\leq M_j\|a\|_{A_j}\|b\|_{B_j},\quad a\in A_0\cap A_1,\quad b\in B_0\cap B_1,\quad j=0,1,
\end{equation*}
where $M_j$ are positive constants.

\begin{theorem}[\cite{CFCM2020}, Theorem 3.2]\label{thm:CFCM}
Let $\bar{A}=(A_0,A_1)$, $\bar{B}=(B_0,B_1)$ be Banach couples. Let $\bar{E}=(E_0,E_1)$ be a couple of Banach function spaces on a $\sigma$-finite measure space $(\Omega,\mu)$ having the Fatou property. Let $T\in{\bf\mathcal{B}}(\bar{A}\times\bar{B},\bar{E})$. If $T:A_1^{\circ}\times B_1^{\circ}\to E_1$ compactly and $E_1$ has absolutely continuous norm, then for any $0<\theta<1$, $T$ may be uniquely extended to a compact bilinear operator from $[A_0,A_1]_\theta\times[B_0,B_1]_\theta$ to $[E_0,E_1]_\theta$.
\end{theorem}

Several remarks are in order.

\begin{remark}
    \item[(a)] In the original statement of \cite[Theorem 3.2]{CFCM2020}, the compactness property was assumed for $T:A_0^{\circ}\times B_0^{\circ}\to E_0$. However, because of the symmetry of the interpolation operator (see for example \cite[Theorem 4.2.1]{BeLo1976}), this is equivalent to the formulation given here.

    \item[(b)] The definition of Banach function spaces in \cite{CFCM2020} includes more assumptions than our own. However, a careful examination of the very detailed exposition of the proof of \cite[Theorem 3.2]{CFCM2020} in \cite{Kouts2021} (and the references therein) shows that the definition of Banach function spaces assumed here is sufficient.

    In particular, let us remark that whenever one needs a sequence $(A_k)^{\infty}_{k=1}$ of measurable sets having finite measure such that $\chi_{A_k}$ lies in the considered Banach space, for al $k=1,2,\ldots$ and $\displaystyle\bigcup_{k=1}^{\infty}A_k=\Omega$, such as in the proof of \cite[Proposition 3.3.3.5]{Kouts2021}, one can use Lemma~\ref{lm:Saturation_Exhaustion}. That is, one does not need to assume that the characteristic function of any measurable set of finite measure lies in the Banach space.
\end{remark}

\begin{remark}\label{rem:Banach_couple}
For $j=0,1$, let $p_j(\cdot)\in\mathscr{P}$ and $w_j$ be weights on $\R^n$, $j=0,1$. Then, by Proposition~\ref{prop:embedding_variable_Lebesgue} we have that $(L^{p_0(\cdot)}(w_0),L^{p_1(\cdot)}(w_1))$ is a Banach couple.
\end{remark}

\section{Proof of abstract main results}\label{sec: pf. main result}

In this section we finalize the proofs of our abstract main results, Theorems~\ref{thm:main_result} and~\ref{thm:main_result_limited_range}.

The last ingredient we need is the following interpolation result which can be found in \cite{Medalha2020} (see also \cite[Theorem 3.2]{FKM2022}).

\begin{theorem}[\cite{Medalha2020}, Theorems 3.4.1 and 3.4.2]\label{thm:SW}
Let $0<\theta<1$. For $j=0,1$, let $p_j(\cdot)\in\mathscr{P}$ and $w_j$ be weights on $\R^n$, $j=0,1$. Then 
\begin{equation*}
    [L^{p_0(\cdot)}(w_0),L^{p_1(\cdot)}(w_1)]_{\theta}=L^{p_\theta(\cdot)}(w_{\theta}),
\end{equation*}
where 
\begin{equation*}
    \frac{1}{p_{\theta}(\cdot)}=\frac{1-\theta}{p_0(\cdot)}+\frac{\theta}{p_1(\cdot)}\quad\text{and}\quad w_{\theta}=w_0^{1-\theta}w_1^{\theta}.
\end{equation*}
Moreover, the norms of the spaces $[L^{p_0(\cdot)}(w_0),L^{p_1(\cdot)}(w_1)]_{\theta}$ and $L^{p_\theta(\cdot)}(w_{\theta})$ are equivalent. 
\end{theorem}

We make some remarks about Theorem~\ref{thm:SW}:

\begin{remark}
    \item[(a)] From Remark~\ref{rem:Banach_couple} we know that $(L^{p_0(\cdot)}(w_0),L^{p_1(\cdot)}(w_1))$ is a Banach couple.
    
    \item[(b)] The statement of \cite[Theorem 3.4.1]{Medalha2020} includes also the assumptions that $w_j\in L_{\text{loc}}^{p_{j}(\cdot)}$ and $w_j^{-1}\in L_{\text{loc}}^{p_{j}'(\cdot)}$. However, in view of Lemma~\ref{lm:Weighted_Variable_Lebesgue_BFS}, a careful examination of the proof given there shows that this assumption is superfluous.

    \item[(c)] Also, the statement of \cite[Theorem 3.4.2]{Medalha2020} includes the assumptions that $w_j\in L_{\text{loc}}^{p_{j}(\cdot)}$ and $w_j^{-1}\in L_{\text{loc}}^{p_{j}'(\cdot)}$. However, as commented in the proof of \cite[Theorem 3.2]{FKM2022}, the only property used in the proof is that weighted variable Lebesgue spaces are Banach lattices, which in view of Lemma~\ref{lm:Weighted_Variable_Lebesgue_BFS} does not require these extra assumptions.
\end{remark}

We have now all ingredients to prove Theorems~\ref{thm:main_result} and~\ref{thm:main_result_limited_range}.

\begin{proof}[Proof of Theorem~\ref{thm:main_result}]
Let $(\vec{p}(\cdot), q(\cdot),\vec{1},\infty)$ be a proper $2$-admissible quadruple such that
\begin{equation*}
    q_{-} > 1,
\end{equation*}
\begin{equation*}
    t<\min\{(p_j)_{-}:~1,2\}
\end{equation*}
and
\begin{equation*}
    \frac{1}{p(\cdot)}-\frac{1}{q(\cdot)}=\gamma.
\end{equation*}
Let also $\vec{w}$ be a 2-tuple of weights with $\vec{w}\in\mathcal{A}_{\vec{p}(\cdot),q(\cdot)}$ and $\vec{w}^{t}\in\mathcal{A}_{\frac{\vec{p}(\cdot)}{t},\frac{q(\cdot)}{t}}$. We will show that $T:L^{p_1(\cdot)}(w_1)\times L^{p_2(\cdot)}(w_2)\to L^{q(\cdot)}(\nu_{\vec{w}})$ compactly.

To see that, we first apply Lemma~\ref{lem:KeyLemma_with_t}: It delivers $\theta>0$ and a proper $2$-admissible quadruple $(\vec{p}_0(\cdot), q_0(\cdot),\vec{1},\infty)$, such that
\begin{equation*}
    (q_0)_{-} > 1,
\end{equation*}
\begin{equation*}
    t<\min\{(p_0,j)_{-}:~j=1,2\},
\end{equation*}
\begin{equation*}
\frac{1}{\vec{p}(\cdot)}=\frac{1-\theta}{\vec{p}_0(\cdot)}+\frac{\theta}{\vec{p}_1(\cdot)},\quad \frac{1}{q(\cdot)}=\frac{1-\theta}{q_0(\cdot)}+\frac{\theta}{q_1(\cdot)}
\end{equation*}
and
\begin{equation*}
    \quad\frac{1}{p_0(\cdot)}-\frac{1}{q_0(\cdot)}=\gamma,
\end{equation*}
as well as a 2-tuple of weights $\vec{w}_0$ with
\begin{equation*}
    \vec{w}=\vec{w}_0^{1-\theta}\vec{w}_1^{\theta}    
\end{equation*}
such that $\vec{w}_0\in\mathcal{A}_{\vec{p}_0(\cdot),q_0(\cdot)}$ and $\vec{w}_0^{t}\in\mathcal{A}_{\frac{\vec{p}_0(\cdot)}{t},\frac{q_0(\cdot)}{t}}$.
   
Set now
\begin{equation*}
    A_j:=L^{p_{j,1}(\cdot)}(w_{j,1}),\quad B_j:=L^{p_{j,2}(\cdot)}(w_{j,2}),\quad E_j:=L^{q_j(\cdot)}(\nu_{\vec{w}_j}),\quad j=0,1.
\end{equation*}
By Lemma~\ref{lm:Weighted_Variable_Lebesgue_BFS} and Remark~\ref{rem:Banach_couple} we have that $A_j,B_j,E_j$, $j=0,1$ satisfy all assumptions of Theorem~\ref{thm:CFCM}. Moreover, since by assumption the operator $T:A_1\times B_1\to E_1$ is compact and $A_1^{\circ}$, respectively $B_1^{\circ}$, is a subspace of $A_1$, respectively $B_1$, we easily deduce from the equivalent characterizations of compactness of bilinear operators (see for instance \cite[Proposition 1]{BenTor2013}) that also the restriction $T:A_1^{\circ}\times B_1^{\circ}\to E_1$ is compact. Therefore, Theorem~\ref{thm:CFCM} yields
\begin{equation*}
    T:[A_0,A_1]_{\theta}\times[B_0,B_1]_{\theta}\to[E_0,E_1]_{\theta}\quad\text{compactly}.
\end{equation*}
Applying Theorem~\ref{thm:SW} we finally deduce
\begin{equation*}
    T:L^{p_{1}(\cdot)}(w_{1})\times L^{p_{2}(\cdot)}(w_{2})\to L^{q(\cdot)}(\nu_{\vec{w}})\quad\text{compactly}.
\end{equation*}
\end{proof}

\begin{proof}[Proof of Theorem~\ref{thm:main_result_limited_range}] 
Let $(\vec{p}(\cdot), q(\cdot),\vec{r},s)$ be a proper $2$-admissible quadruple such that
\begin{equation*}
    q_{-} > 1
\end{equation*}
and
\begin{equation*}
    \frac{1}{p(\cdot)}-\frac{1}{q(\cdot)}=\gamma.
\end{equation*}
Let also $\vec{w}$ be a 2-tuple of weights with $\vec{w}\in\mathcal{A}_{(\vec{p}(\cdot),q(\cdot),\vec{r},s)}$. We will show that $T:L^{p_1(\cdot)}(w_1)\times L^{p_2(\cdot)}(w_2)\to L^{q(\cdot)}(\nu_{\vec{w}})$ compactly.

To see that, we first apply Lemma~\ref{lem:KeyLemma}: It delivers $\theta>0$ and a proper $2$-admissible quadruple $(\vec{p}_0(\cdot), q_0(\cdot),\vec{r},s)$, such that
\begin{equation*}
    (q_0)_{-} > 1,
\end{equation*}
\begin{equation*}
    \frac{1}{\vec{p}(\cdot)}=\frac{1-\theta}{\vec{p}_0(\cdot)}+\frac{\theta}{\vec{p}_1(\cdot)},\quad \frac{1}{q(\cdot)}=\frac{1-\theta}{q_0(\cdot)}+\frac{\theta}{q_1(\cdot)}
    \end{equation*}
and
\begin{equation*}
    \quad\frac{1}{p_0(\cdot)}-\frac{1}{q_0(\cdot)}=\gamma,
\end{equation*}
as well as a 2-tuple of weights $\vec{w}_0$ with
\begin{equation*}
    \vec{w}=\vec{w}_0^{1-\theta}\vec{w}_1^{\theta}    
\end{equation*}
such that $\vec{w}_0\in\mathcal{A}_{(\vec{p}_0(\cdot),q_0(\cdot),\vec{r},s)}$.

The rest of the proof is similar as the one of Theorem~\ref{thm:main_result}. 
\end{proof}

\section{Applications}\label{sec: applications}

In this section, we provide several applications of our compact extrapolation Theorems \ref{thm:main_result} and \ref{thm:main_result_limited_range} for bilinear operators. In particular, we will show the compactness of commutators of bilinear $\omega$-Calder\'on--Zygmund operators, bilinear fractional integral operators and bilinear Fourier multipliers.

\subsection{Notation for multilinear commutators}

Given some $m$-linear operator $T$ accepting as arguments $m$-tuples $\vec{f}=(f_1,\ldots,f_m)$ of appropriate functions on $\R^n$ and some appropriately smooth function $b$ on $\R^n$, we define for each $j=1,\ldots,m$ the commutator
\begin{equation*}
    [T,b]_{e_j}(\vec{f})(x) :=b(x)T(\vec{f})(x)-T(f_1, \ldots,bf_j,\ldots,f_m)(x).
\end{equation*}

As we will see below, in order to apply our extrapolation Theorem~\ref{thm:main_result} we require that the pointwise multiplier $b$ belongs to the spaces
\begin{equation*}
    \BMO(\R^n):=\Big\{f:\R^n\to\C\ \Big|\ \Norm{f}{\BMO(\R^n)}:=\sup_{Q\subset\R^n}\ave{\abs{f-\ave{f}_Q}}_Q<\infty\Big\}
\end{equation*}
or
\begin{equation*}
    \CMO(\R^n):=\overline{C_c^\infty(\R^n)}^{\BMO(\R^n)}.
\end{equation*}

Moreover, given a $m$-tuple $\vec{b}=(b_1,\ldots,b_m)$ of appropriately smooth functions on $\R^n$, we consider the \emph{sum commutator}
\begin{equation*}
    [T,\vec{b}]_{\Sigma}(\vec{f}) := \sum_{j=1}^{m}[T,b_j]_{e_j}(\vec{f}).
\end{equation*}

\subsection{Notation for maximal operators}

For $\delta,\varepsilon,\alpha>0$ and $\vec{r}=(r_1,\ldots,r_m)\in(1,\infty)^m$ we consider the maximal operators
\begin{equation*}
    M^{\#}_{\delta}(f)(x) := \sup_{\substack{Q\text{ cube in }\R^n\\x\in Q}}\left(\frac{1}{|Q|}\int_{Q}||f(y)|^{\delta}-\langle |f|^{\delta}\rangle_{Q}|\,\mathrm{d}y\right)^{1/\delta},
\end{equation*}
\begin{equation*}
    M_{\varepsilon}(f)(x) := \sup_{\substack{Q\text{ cube in }\R^n\\x\in Q}}\left(\frac{1}{|Q|}\int_{Q}|f(y)|^{\varepsilon}\,\mathrm{d}y\right)^{1/\varepsilon},
\end{equation*}
\begin{equation*}
    \mathcal{M}_{\alpha}(\vec{f})(x) := \sup_{\substack{Q\text{ cube in }\R^n\\x\in Q}}\frac{1}{|Q|^{m-\frac{\alpha}{n}}}\prod_{j=1}^{m}\int_{Q}|f_j(y_j)|\mathrm{d}y_j,
\end{equation*}
\begin{equation*}
    \mathcal{M}(\vec{f})(x) :=\sup_{\substack{Q\text{ cube in }\R^n\\x\in Q}}\prod_{j=1}^{m}\frac{1}{|Q|}\int_{Q}|f_j(y_j)|\mathrm{d}y_j,
\end{equation*}
    and
\begin{equation*}
    \mathcal{M}_{\alpha,\mathrm{L(log L)}}(\vec{f})(x) := \sup_{\substack{Q\text{ cube in }\R^n\\x\in Q}}|Q|^{\alpha/n}\prod_{k=1}^{m}\Vert f_k\Vert_{\mathrm{L(log L),Q}}, 
\end{equation*}
    where $\Vert f\Vert_{\mathrm{L(log L),Q}}$ is the usual $\mathrm{L(\log L)}$ Luxemburg norm of $f\mathbf{1}_{Q}$ over $Q$ with respect to normalized Lebesgue measure on $Q$, concretely
\begin{equation*}
    \Vert f\Vert_{\mathrm{L(log L),Q}} := \inf\left\{\lambda>0:~\frac{1}{|Q|}\int_{Q}\frac{1}{\lambda}\log\left(e+\frac{|f(y)|}{\lambda}\,\mathrm{d}y\right)\leq1\right\}.
\end{equation*}

\subsection{Bilinear \texorpdfstring{$\omega$}{ω}-Calder\'on--Zygmund operators}\label{subsec:CZ}

Let $\omega:[0,\infty)\rightarrow[0,\infty)$ be a modulus of continuity, that is, $\omega$ is increasing, subadditive and $\omega(0)=0$. A function  $K(x,y_1,\ldots,y_m)$ on $(\R^n)^{{m+1}}$ defined away from the diagonal $x=y_1=\ldots=y_m$ is said to be an \emph{$\omega$-Calder\'{o}n--Zygmund kernel} if the following \emph{size and smoothness conditions} are satisfied:    
\begin{align*}
    |K(x,y_1,\ldots,y_m)|&\lesssim \frac{1}{(\sum_{j=1}^m|x-y_j|)^{mn}},\\
    |K(x,y_1,\ldots,y_m)-K(x',y_1,\ldots,y_m)|&\lesssim \frac{1}{(\sum_{j=1}^m|x-y_j|)^{mn}}\omega\bigg(\frac{|x-x'|}{\sum_{j=1}^m|x-y_j|}\bigg),
\end{align*} 
whenever $|x-x'|\leq\frac{1}{2}\max_{j=1,\ldots,m}|x-y_j|$, and for each $i=1,\dots,m,$
\begin{equation*}
    |K(x,y_1,\ldots,y_i,\ldots,y_m)-K(x,y_1,\ldots,y_i',\ldots,y_m)| \lesssim \frac{1}{(\sum_{j=1}^m|x-y_j|)^{mn}}\omega\bigg(\frac{|y_i-y_i'|}{\sum_{j=1}^m|x-y_j|}\bigg),
\end{equation*}
whenever $|y_i-y_i'|\leq\frac{1}{2}\max_{j=1,\ldots,m}|x-y_j|$.

The $m$-linear operator $T:\mathcal{S}(\R^n)\times\cdots\times\mathcal{S}(\R^n)\rightarrow\mathcal{S}'(\R^n)$ is said to be an $\omega$-Calder\'on--Zygmund operator if it extends to a bounded multilinear operator from $L^{q_1}(\R^n)\times\cdots\times L^{q_m}(\R^n)$ to $L^q(\R^n)$ for some $\frac{1}{q}=\frac{1}{q_1}+\cdots+\frac{1}{q_m}$ with $1<q_1,\dots,q_m<\infty$, and there exists an $\omega$-Calder\'on--Zygmund kernel $K$ such that
\begin{equation*}
    T(\vec{f})(x) := \int_{(\R^n)^m}K(x,y_1,\ldots,y_m)f(y_1)\cdots f(y_m)\,\mathrm{d}y_1\cdots\mathrm{d}y_m,
\end{equation*}
for all $x\notin\cap_{j=1}^m\supp f_j$, $j=1,\dots,m$.

We recall that a modulus of continuity $\omega$ is said to satisfy the Dini condition, denoted by $\omega\in$ Dini, if 
\begin{equation*}
    \|\omega\|_{\text{Dini}}:=\int_0^1\omega(t)\frac{\mathrm{d}t}{t}<\infty.
\end{equation*}
Notice that the modulus of continuity $\omega(t)=t^\varepsilon$ with $\varepsilon>0$ satisfies the Dini condition. This example appeared in the work of Grafakos and Torres \cite{GT2002}, where they studied the (standard) Calder\'on--Zygmund operator. It is worth to mention that for the general $\omega$, the linear $\omega$-Calder\'on--Zygmund operator was studied by Yabuta in \cite{Yabuta1985}. A further extension of this result was obtained by Maldonado and Naibo \cite{MN2009} in the bilinear case. 

The authors of \cite[Theorem 5.1]{CaoOlivoYabuta2022} proved the compactness of commutators of m-linear $\omega$-Calder\'on--Zygmund operators on weighted Lebesgue spaces  
with constant exponents. We point out that one part of the proof of \cite[Theorem 5.1]{CaoOlivoYabuta2022} relied on showing the following unweighted compactness result.  

\begin{theorem}[\cite{CaoOlivoYabuta2022}]\label{thm:comp_bil_CZO}
Let $b\in\mathrm{CMO}(\R^n)$, $1<p_1,\ldots,p_m<\infty$ and $0<p<\infty$ with $\frac{1}{p}:=\sum_{j=1}^{m}\frac{1}{p_j}$. Then, $[T,b]_{e_j}$ maps $L^{p_1}(\R^n)\times\cdots\times L^{p_m}(\R^n)$ into $L^{p}(\R^n)$ compactly for every $j=1,\ldots,m$.    
\end{theorem}

We will extend Theorem \ref{thm:comp_bil_CZO} to the weighted variable Lebesgue spaces in the bilinear setting. To this end, we need the following two weighted boundedness results. 

\begin{theorem}\label{thm:bound_bil_CZO_1}
If $T$ is a bilinear $\omega$-Calder\'on--Zygmund operator, then for all proper $2$-admissible quadruples $(\vec{p}(\cdot),\vec{p}(\cdot),\vec{1},\infty)$ and for all pairs of weights $\vec{w}$ with $\vec{w}\in\mathcal{A}_{\vec{p}(\cdot)}$, it holds that $T$ maps $L^{p_1(\cdot)}(w_1)\times\cdots\times L^{p_m(\cdot)}(w_m)$ into $L^{p(\cdot)}(\nu_{\vec{w}})$ boundedly.
\end{theorem}

\begin{proof}
The proof is almost the same to the one given in \cite[Theorem 2.8]{CG2020}. In particular, the main difference is that instead of the pointwise estimate of \cite[Proposition 7.2]{CG2020} we make use of the following one (see \cite[Theorem 6.1]{LuZhang2014}) 
\begin{equation*}
    M^{\#}_{\delta}(T(\vec{f}))(x)\lesssim\mathcal{M}(\vec{f})(x),    
\end{equation*} 
where $0<\delta<1/m$.
\end{proof}

\begin{remark}
We notice that one of the main ingredients of the proof of \cite[Theorem 2.8]{CG2020} was the characterization of $\mathcal{A}_{\vec{p}(\cdot)}$ weights by the bilinear maximal operator $\mathcal{M}$ (see \cite[Theorem 2.4]{CG2020}). According to \cite[Remark 2.7]{CG2020}, one can obtain the multilinear version of this latter result. We also refer to the work \cite[Theorem 2.2]{CHSW2025}, in which the characterization of the $\mathcal{A}_{\vec{p}(\cdot),q(\cdot)}$ weights was established with the aid of the multilinear fractional maximal operator $\mathcal{M}_{\alpha}$.
\end{remark}

\begin{theorem}\label{thm:bound_bil_CZO_2}
Let $b\in\mathrm{BMO}(\R^n)$. Let $t$ be a fixed constant with $t>1$. Let $j\in\{1,\ldots,m\}$. Then, for all proper $m$-admissible quadruples $(\vec{p}(\cdot),\vec{p}(\cdot),\vec{1},\infty)$ with
\begin{equation*}
    t<\min\{(p_k)_{-}:~k=1,\ldots,m\},
\end{equation*}
and for all $m$-tuples of weights $\vec{w}$ with $\vec{w}\in\mathcal{A}_{\vec{p}(\cdot)}$ and $\vec{w}^{t}\in\mathcal{A}_{\frac{\vec{p}(\cdot)}{t}}$, it holds that $[T,b]_{e_j}$ maps $L^{p_1(\cdot)}(w_1)\times\cdots\times L^{p_m(\cdot)}(w_m)$ into $L^{p(\cdot)}(\nu_{\vec{w}})$ boundedly.
\end{theorem}

\begin{proof}
The proof is similar to \cite[Theorem 2.4]{CHSW2025}. Here, we provide the necessary changes. In particular, let $0<\delta<\varepsilon<1/m$. Then, by \cite[page 95 and Theorem 7.1]{LuZhang2014} we have the pointwise estimate
\begin{align*}
    M^{\#}_{\delta}([T,b]_{e_j}(\vec{f}))(x)&\lesssim\Vert b\Vert_{\mathrm{BMO}(\R^n)}(M_{\varepsilon}(T(\vec{f}))(x)+\mathcal{M}_{\mathrm{L(log(L))}}(\vec{f})(x))\\
    &\lesssim\Vert b\Vert_{\mathrm{BMO}(\R^n)}(M_{\frac{1}{m}}(T(\vec{f}))(x)+\mathcal{M}_{\mathrm{L(log(L))}}(\vec{f})(x)),
\end{align*}
where the last inequality follows by using Jensen's inequality. 
The rest of the proof proceeds in a similar fashion as in \cite[Theorem 2.4]{CHSW2025} with the difference that instead of \cite[Theorem 2.3]{CHSW2025} we apply Theorem \ref{thm:bound_bil_CZO_1}; we skip the details.    
\end{proof}

Now, we are in position to apply our abstract main result Theorem~\ref{thm:main_result} which gives the following.

\begin{theorem}\label{thm:main_result_CZO}
Let $\vec{b}=(b_1,b_2)$ be a pair of functions in $\mathrm{CMO}(\R^n)$ and $t$ be a fixed constant with $t>1$. Then, for all proper $2$-admissible quadruples $(\vec{p}(\cdot),\vec{p}(\cdot),\vec{1},\infty)$ with
\begin{equation*}
    t<\min\{(p_j)_{-}:~j=1,2\},
\end{equation*}
and for all pairs of weights $\vec{w}$ with $\vec{w}\in\mathcal{A}_{\vec{p}(\cdot)}$ and $\vec{w}^{t}\in\mathcal{A}_{\frac{\vec{p}(\cdot)}{t}}$, it holds that $[T,b_1]_{e_1}$, $[T,b_2]_{e_2}$ and $[T,\vec{b}]_{\Sigma}$ map $L^{p_1(\cdot)}(w_1)\times L^{p_2(\cdot)}(w_2)$ into $L^{p(\cdot)}(\nu_{\vec{w}})$ compactly.
\end{theorem}

\begin{proof}
The first step is to apply Theorem~\ref{thm:main_result} for the operators $[T,b_1]_{e_1}$ and $[T,b_2]_{e_2}$, noticing that $\gamma=0$.
    
Part \eqref{eq:main1} of the assumptions of Theorem~\ref{thm:main_result} for $[T,b_1]_{e_1}$ and $[T,b_2]_{e_2}$ follows from the conclusion of Theorem~\ref{thm:bound_bil_CZO_2} in the bilinear case (with an obvious change of notation).
    
Choose now a constant
\begin{equation*}
    x_1>\max\{t,2\}
\end{equation*}
Such a choice is always possible. Then, Theorem~\ref{thm:comp_bil_CZO} implies that part \eqref{eq:main2} of the assumptions of Theorem~\ref{thm:main_result} is satisfied with the particular choices $\vec{p}_1(\cdot):=(x_1,x_1)$, $p_1(\cdot):=\frac{x_1}{2}$ and $\vec{w}_1:=(1,1)$ for each of the operators $[T,b_1]_{e_1}$ and $[T,b_2]_{e_2}$.

Hence, by Theorem~\ref{thm:main_result} we deduce the desired conclusion for each of the operators $[T,b_1]_{e_1}$ and $[T,b_2]_{e_2}$. This implies immediately the desired conclusion also for $[T,\vec{b}]_{\Sigma}$, since the sum of bilinear compact operators remains compact due to \cite[Proposition 2]{BenTor2013}.
\end{proof}

We close this subsection by mentioning  
that there are several examples of the bilinear $\omega$-Calder\'on--Zygmund operators (see \cite[Subsection 5.1]{CaoOlivoYabuta2022}) which satisfy the assumptions and conclusion of Theorem \ref{thm:main_result_CZO}.

\subsection{Bilinear fractional Calder\'on--Zygmund operators}\label{subsec:fractional}

Let $m\geq2$ and $\alpha\in(0,mn)$. A function $K_{\alpha}(x,y_1,\ldots,y_m)$ on $(\R^n)^{m+1}$ defined away from the diagonal $x=y_1=\ldots=y_m$ is said to be a \emph{$m$-linear $\alpha$-fractional Calder\'{o}n--Zygmund kernel} if there exist constants $A,B>0$, such that the following \emph{size estimate} is satisfied:
\begin{equation*}
    |K_{\alpha}(x,y_1,\ldots,y_m)| \leq \frac{A}{(|x-y_1|+\cdots+|x-y_m|)^{mn-\alpha}},
\end{equation*}
as well as the following \emph{smoothness estimates} hold:
\begin{itemize}
\item It holds
\begin{equation*}
    |K_{\alpha}(x,y_1,\ldots,y_m)-K_{\alpha}(x',y_1,\ldots,y_m)| \leq \frac{B|x-x'|}{(|x-y_1|+\cdots+|x-y_m|)^{mn-\alpha+1}},
\end{equation*}
whenever $|x-x'|\leq\frac{1}{2}\max_{j=1,\ldots,m}|x-y_j|$.

\item For each $j=1,\ldots,m$, one has
\begin{equation*}
    | K_{\alpha}(x,y_1,\ldots,y_j,\ldots,y_m)-K_{\alpha}(x,y_1,\ldots,y_j',\ldots,y_m)| \leq \frac{B|y_j-y_j'|}{(|x-y_1|+\cdots+|x-y_m|)^{mn-\alpha+1}},
\end{equation*}
whenever $|y_j-y_j'|\leq\frac{1}{2}\max_{k=1,\ldots,m}|x-y_k|$.
\end{itemize}

As noted in \cite{BDMT2015} in the special case $m=2$, given such a kernel $K_{\alpha}$, the $m$-linear operator $T_{\alpha}$ defined by
\begin{equation*}
     T_{\alpha}(\vec{f})(x) := \int_{(\R^n)^m}K_{\alpha}(x,y_1,\ldots,y_m)f_1(y_1)\cdots f_m(y_m)\,\mathrm{d}y_1\cdots\mathrm{d}y_m,\quad x\in\R^n
\end{equation*}
is well defined whenever $f_j$, $j=1,\ldots,m$ are for example bounded functions with compact support on $\R^n$.

As observed in \cite[page 3]{BDMT2015} in the special case $m = 2$ (see also \cite[Theorem 5.8]{CaoOlivoYabuta2022}), an example of such an operator is the \emph{$m$-linear fractional integral operator} $\mathcal{I}_{\alpha}$, also known as the \emph{$m$-linear Riesz potential}, associated to the kernel
\begin{equation*}
    K_{\alpha}(x,y_1,\ldots,y_m) := \frac{1}{(|x-y_1|+\cdots+|x-y_m|)^{mn-\alpha}}.
\end{equation*}
In fact, for every such operator $T_{\alpha}$ one has by the size estimate of the kernel
\begin{equation}\label{trivial_bound}
    |T_{\alpha}(\vec{f})(x)| \leq A\,\mathcal{I}_{\alpha}(|f_1|,\ldots,|f_m|)(x).
\end{equation}
It was shown in \cite{KenigStein1999} that $\mathcal{I}_{\alpha}$ extends to a bounded $m$-linear operator $L^{p_1}(\R^n)\times\cdots\times L^{p_m}(\R^n)\to L^{q}(\R^n)$ whenever $1<p_1,\ldots,p_m<\infty$, $0<q<\infty$ and $\frac{\alpha}{n}=\frac{1}{p}-\frac{1}{q}$, where $\frac{1}{p}:=\sum_{j=1}^{m}\frac{1}{p_j}$. Thus, the same holds for arbitrary $T_\alpha$.

A result on the compactness of commutators of $T_{\alpha}$ on unweighted Lebesgue spaces in the bilinear case was first established in \cite[Theorem 2.1]{BDMT2015}. Subsequently, \cite[Theorem 5.8]{CaoOlivoYabuta2022} extended \cite[Theorem 2.1]{BDMT2015} to the general $m$-linear setting and (constant exponent) weighted Lebesgue spaces and explicitly allowed the target space to be only a quasi-Banach space. Let us note that while \cite[Theorem 5.8]{CaoOlivoYabuta2022} is formally stated only for $\mathcal{I}_{\alpha}$, the authors of \cite{CaoOlivoYabuta2022} observe that the statement and their proof hold for arbitrary  $T_{\alpha}$. Here we only need the unweighted version of \cite[Theorem 5.8]{CaoOlivoYabuta2022}.

\begin{theorem}[\cite{BDMT2015, CaoOlivoYabuta2022}]\label{thm:comp_bil_frac}
Let $b\in\mathrm{CMO}(\R^n)$. Let $1<p_1,\ldots,p_m<\infty$ and let $0<q<\infty$ such that $\frac{1}{p}-\frac{1}{q}=\frac{\alpha}{n}$, where $\frac{1}{p}:=\sum_{j=1}^{m}\frac{1}{p_j}$. Then, $[T_{\alpha},b]_{e_j}$ maps $L^{p_1}(\R^n)\times\cdots\times L^{p_m}(\R^n)$ into $L^{q}(\R^n)$ compactly for every $j=1,\ldots,m$.
\end{theorem}

Our goal consists in generalizing Theorem~\ref{thm:comp_bil_frac} to weighted variable exponent Lebesgue spaces in the bilinear case. To achieve that, we first need a boundedness result over such spaces. For the special case of the $m$-Riesz potential, this is just \cite[Theorem 2.4]{CHSW2025}. Our proof of the general case is only a slight adaptation of \cite[Theorem 2.4]{CHSW2025}.

\begin{theorem}\label{thm:bound_bil_frac}
Let $b\in\mathrm{BMO}(\R^n)$. Set $\gamma:=\frac{\alpha}{n}$. Let $t$ be a fixed constant with $1<t<\frac{m}{\gamma}$. Let $j\in\{1,\ldots,m\}$. Then, for all proper $2$-admissible quadruples $(\vec{p}(\cdot), q(\cdot),\vec{1},\infty)$ with
\begin{equation*}
    t<\min\{(p_k)_{-}:~k=1,\ldots,m\}
\end{equation*}
and
\begin{equation*}
    \frac{1}{p(\cdot)}-\frac{1}{q(\cdot)}=\gamma,
\end{equation*}
and for all $m$-tuples of weights $\vec{w}$ with $\vec{w}\in\mathcal{A}_{\vec{p}(\cdot),q(\cdot)}$ and $\vec{w}^{t}\in\mathcal{A}_{\frac{\vec{p}(\cdot)}{t},\frac{q(\cdot)}{t}}$, it holds that $[T_{\alpha},b]_{e_j}$ maps $L^{p_1(\cdot)}(w_1)\times\cdots\times L^{p_m(\cdot)}(w_m)$ into $L^{q(\cdot)}(\nu_{\vec{w}})$ boundedly.
\end{theorem}

\begin{proof}
The proof is almost identical to \cite[Theorem 2.4]{CHSW2025}, so we explain only the necessary changes.
    
First of all, by \cite[Theorem 2.3]{CHSW2025} we have that the $m$-linear Riesz potential $\mathcal{I}_{\alpha}$ maps $L^{p_1(\cdot)}(w_1)\times\cdots\times L^{p_m(\cdot)}(w_m)$ into $L^{q(\cdot)}(\nu_{\vec{w}})$ boundedly. Thus, by \eqref{trivial_bound} we have that the same holds for arbitrary $T_{\alpha}$.
    
Let now $0<\delta<\varepsilon<\frac{n}{mn-\alpha}$. Then, by \cite[Lemma 2.24]{WZ2023} we have the pointwise estimate
\begin{equation*}
    M^{\#}_{\delta}([T_{\alpha},b]_{e_j}(\vec{f}))(x)\leq C\Vert b\Vert_{\mathrm{BMO}(\R^n)}(M_{\varepsilon}(T_{\alpha}(\vec{f}))(x)+\mathcal{M}_{\alpha,\mathrm{L(log(L))}}(\vec{f})(x)),
\end{equation*}
where the constant $C$ depends only on $n, m$, the constants of $T_{\alpha}$, $\varepsilon$ and $\delta$. We now pick $\varepsilon := \frac{1}{m}$, which is clearly a permitted choice. The rest of the proof proceeds exactly as in \cite[Theorem 2.4]{CHSW2025}; we omit the details.
\end{proof}

Applying now our abstract main result, Theorem~\ref{thm:main_result}, we deduce the following.

\begin{theorem}
Let $\vec{b}=(b_1,b_2)$ be a pair of functions in $\mathrm{CMO}(\R^n)$. Set $\gamma:=\frac{\alpha}{n}$. Let $t$ be a fixed constant with $1<t<\frac{2}{\gamma}$. Then, for all proper $2$-admissible quadruples $(\vec{p}(\cdot),q(\cdot),\vec{1},\infty)$ with
\begin{equation*}
    q_{-} > 1,
\end{equation*}
\begin{equation*}
    t<\min\{(p_j)_{-}:~j=1,2\}
\end{equation*}
and
\begin{equation*}
    \frac{1}{p(\cdot)}-\frac{1}{q(\cdot)}=\gamma
\end{equation*}
and for all pairs of weights $\vec{w}$ with $\vec{w}\in\mathcal{A}_{\vec{p}(\cdot),q(\cdot)}$ and $\vec{w}^{t}\in\mathcal{A}_{\frac{\vec{p}(\cdot)}{t},\frac{q(\cdot)}{t}}$, it holds that $[T_{\alpha},b_1]_{e_1}$, $[T_{\alpha},b_2]_{e_2}$ and $[T_{\alpha},\vec{b}]_{\Sigma}$ map $L^{p_1(\cdot)}(w_1)\times L^{p_2(\cdot)}(w_2)$ into $L^{q(\cdot)}(\nu_{\vec{w}})$ compactly.
\end{theorem}

\begin{proof}
We first apply Theorem~\ref{thm:main_result} for the operators $[T_{\alpha},b_1]_{e_1}$ and $[T_{\alpha},b_2]_{e_2}$, noting that $\gamma<\frac{2}{t}$.
    
Part \eqref{eq:main1} of the assumptions of Theorem~\ref{thm:main_result} for $[T_{\alpha},b_1]_{e_1}$ and $[T_{\alpha},b_2]_{e_2}$ is just the conclusion of Theorem~\ref{thm:bound_bil_frac} in the bilinear case (with different notation).
    
Pick now a constant
\begin{equation*}
    d_1 \in \left(\max\left\{t,\frac{2}{1+\gamma}\right\},\frac{2}{\gamma}\right).
\end{equation*}
Such a choice is clearly possible. Then, Theorem~\ref{thm:comp_bil_frac} yields that part \eqref{eq:main2} of the assumptions of Theorem~\ref{thm:main_result} is satisfied with the particular choices $\vec{p}_1(\cdot):=(d_1,d_1)$, $q_1(\cdot):=\frac{1}{\frac{2}{d_1}-\gamma}$ and $\vec{w}_1:=(1,1)$ for each of the operators $[ T_{\alpha},b_1]_{e_1}$ and $[T_{\alpha},b_2]_{e_2}$.

    Thus, by Theorem~\ref{thm:main_result} we deduce the desired conclusion for each of the operators $[T_{\alpha},b_1]_{e_1}$ and $[T_{\alpha},b_2]_{e_2}$. This yields immediately the desired conclusion also for $[T_{\alpha},\vec{b}]_{\Sigma}$, since the sum of bilinear compact operators remains compact per \cite[Proposition 2]{BenTor2013}.
\end{proof}

\subsection{Bilinear Fourier multipliers}\label{subsec:fourier}

Let $s\in\N$ and $m\geq 2$. We consider a bounded function $\mathfrak{m}\in\mathrm{C}^{s}(\R^{nm}\setminus\{0\})$. Let $\Phi$ be a Schwarz function on $\R^{nm}$ satisfying the conditions
\begin{equation*}
    \mathrm{supp}(\Phi) \subseteq \left\{(\xi_1,\ldots,\xi_{m})\in(\R^{n})^{m}:~\frac{1}{2}\leq\sum_{j=1}^{m}|\xi_j|\leq 2\right\}
\end{equation*}
and
\begin{equation*}
    \sum_{k\in\Z}\Phi(2^{-j}\xi)=1,\quad\forall\; \xi = (\xi_1,\ldots,\xi_m)\in\R^{nm}\setminus\{0\}.
\end{equation*}
Consider the usual Sobolev space $W^{s}(\R^{nm})$ with norm
\begin{equation*}
    \Vert f\Vert_{W^{s}(\R^{nm})} := \left(\int_{\R^{nm}}(1+|\xi|^2)^{s}|\widehat{f}(\xi)|^2\,\mathrm{d}\xi\right)^{1/2},
\end{equation*}
where $\widehat{f}$ denotes the Fourier transform of $f$ on all the variables in $\R^{nm}$. Define
\begin{equation*}
    \mathfrak{m}_{k}(\xi) := \Phi(\xi)\mathfrak{m}(2^{k}\xi),\quad \xi\in\R^{nm}
\end{equation*}
for each $k\in\Z$. We say that $\mathfrak{m}\in\mathcal{W}^{s}(\R^{nm})$ if
\begin{equation*}
    \Vert\mathfrak{m}\Vert_{\mathcal{W}^{s}(\R^{nm})} := \sup_{k\in\Z}\Vert\mathfrak{m}_{k}\Vert_{W^{s}(\R^{nm})}<\infty.
\end{equation*}
Now consider the $m$-linear Fourier multiplier $T_{\mathfrak{m}}$ associated to the symbol $\mathfrak{m}$ defined by
\begin{equation*}
    T_{\mathfrak{m}}(f_1,\ldots,f_m)(x) := \int_{(\R^n)^m}\mathfrak{m}(\xi)e^{2\pi i x\cdot(\xi_1+\cdots+\xi_m)}\widehat{f}_1(\xi_1)\cdots\widehat{f}_m(\xi_m)\,\mathrm{d}\xi,\quad x\in\R^n
\end{equation*}
for Schwarz functions $f_1,\ldots,f_m$ on $\R^n$.

The compactness properties of the commutator $[T_{\mathfrak{m}},b]_{e_j}$, $j=1,\ldots,m$ on (constant exponent) weighted Lebesgue spaces were first considered in the bilinear case in \cite{Hu2017} and later extended to the general $m$-linear case in \cite[Theorem 5.9]{CaoOlivoYabuta2022}. Here we need only the unweighted version of \cite[Theorem 5.9]{CaoOlivoYabuta2022}.

\begin{theorem}[{\cite[Theorem 5.9]{CaoOlivoYabuta2022}}]
\label{thm:comp_FourMult_com_unw}
Assume $\frac{mn}{2}<s\leq mn$ and $\mathfrak{m}\in\mathcal{W}^{s}(\R^{nm})$. Let $1\leq t_1,\ldots, t_m <2$ with
\begin{equation*}
    \frac{s}{n} = \sum_{j=1}^{m}\frac{1}{t_j}.
\end{equation*}
Let $b\in\mathrm{CMO}(\R^{n})$. Then, for all $1<p_1,\ldots,p_m<\infty$ with $t_j<p_j$ for each $j=1,\ldots,m$ and
\begin{equation*}
    \frac{1}{p} := \sum_{j=1}^{m}\frac{1}{p_j},
\end{equation*}
we have that $[T_{\mathfrak{m}},b]_{e_j}$ maps $L^{p_1}(\R^n)\times\cdots\times L^{p_m}(\R^n)$ into $L^{p}(\R^n)$ compactly for each $j=1,\ldots,m$.
\end{theorem}

Our goal consists in proving a version of \cite[Theorem 5.9]{CaoOlivoYabuta2022} for variable exponent weighted Lebesgue spaces in the bilinear case. To achieve that, we first to study the boundedness over such spaces. To that end, we adapt the proof of \cite[Theorem 2.4]{CHSW2025}, using some estimates from \cite{Li_Sun_2013}. We first need to study the boundedness of $T_{\mathfrak{m}}$ itself.

\begin{theorem}\label{thm:bound_FourMult}
Assume $\frac{mn}{2}<s\leq mn$ and $\mathfrak{m}\in\mathcal{W}^{s}(\R^{nm})$. Then, there exists a constant $p_0\in\left(\frac{mn}{s},\infty\right)$ depending only on $m,n$ and $s$, such that the following holds. For all $r_0\in[p_0,\infty)$, for all proper $m$-admissible quadruples $(\vec{p}(\cdot),\vec{p}(\cdot),\vec{r},\infty)$ with $r_j := r_0$, $j = 1, \ldots, m$, and for all $\vec{w}\in\mathcal{A}_{(\vec{p}(\cdot),\vec{p}(\cdot)),(\vec{r},\infty)}$, we have that $T_{\mathfrak{m}}$ maps $L^{p_1(\cdot)}(w_1)\times\cdots\times L^{p_m(\cdot)}(w_m)$ into $L^{p(\cdot)}(\nu_{\vec{w}})$ boundedly.
\end{theorem}

\begin{proof}
By \cite[Lemma 2.6]{Li_Sun_2013} we have that there exists a constant $p_0\in(1,\infty)$ depending only on $m, n$ and $s$, such that for all bounded compactly supported functions $f_1,\ldots,f_m$ one has
\begin{equation*}
    M^{\#}_{\delta}(T_{\mathfrak{m}}(\vec{f})) \lesssim_{\delta} \mathcal{M}(|f_1|^{p_0},\ldots,|f_m|^{p_0})^{1/p_0},
\end{equation*}
for all $0<\delta<\frac{p_0}{m}$.
    
Now we fix such a constant $0<\delta<\frac{p_0}{m}$. Let $r_0\in[p_0,\infty)$, $(\vec{p}(\cdot),\vec{p}(\cdot),\vec{r},\infty)$ be a proper $m$-admissible quadruple with $r_j := r_0$, $j = 1, \ldots, m$, and let $\vec{w}\in\mathcal{A}_{(\vec{p}(\cdot),\vec{p}(\cdot)),(\vec{r},\infty)}$. Set $w:=\nu_{\vec{w}}$. Since $p_0\leq r_0$, by Jensen's inequality we have
\begin{align*}
    \mathcal{M}(|f_1|^{p_0},\ldots,|f_m|^{p_0})^{1/p_0} \leq \mathcal{M}(|f_1|^{r_0},\ldots,|f_m|^{r_0})^{1/r_0}.
\end{align*}
Thus, using Lemma \ref{lem:homog} we can estimate
\begin{align*}
    &\Vert \mathcal{M}(|f_1|^{p_0},\ldots,|f_m|^{p_0})^{1/p_0}\Vert_{L^{p(\cdot)}(w)}
    \leq \Vert \mathcal{M}(|f_1|^{r_0},\ldots,|f_m|^{r_0})^{1/r_0}\Vert_{L^{p(\cdot)}(w)}\\
    &= \Vert \mathcal{M}(|f_1|^{r_0},\ldots,|f_m|^{r_0})^{1/r_0}w\Vert_{L^{p(\cdot)}}
    = \Vert \mathcal{M}(|f_1|^{r_0},\ldots,|f_m|^{r_0})w^{r_0}\Vert_{L^{p(\cdot)/r_0}}^{1/r_0}\\
    &= \Vert \mathcal{M}(|f_1|^{r_0},\ldots,|f_m|^{r_0})\Vert_{L^{p(\cdot)/r_0}(w^{r_0})}^{1/r_0}.
\end{align*}
Since $\vec{w}\in\mathcal{A}_{(\vec{p}(\cdot),\vec{p}(\cdot)),(\vec{r},\infty)}$, by Lemma~\ref{lem:rescale} we have $\vec{w}^{r_0}\in\mathcal{A}_{(\vec{p}(\cdot)/r_0,\vec{p}(\cdot)/r_0),(\vec{r}/r_0,\infty)}$, that is $\vec{w}^{r_0}\in\mathcal{A}_{\vec{p}(\cdot)/r_0}$. Thus, using the bound for $\mathcal{M}$ from \cite[Theorem 2.4]{CG2020} coupled with \cite[Remark 2.7]{CG2020} we deduce
\begin{align*}
    \Vert \mathcal{M}(|f_1|^{r_0},\ldots,|f_m|^{r_0})\Vert_{L^{p_j(\cdot)/r_0}(w^{r_0})}^{1/r_0} \lesssim \prod_{j=1}^{m}\Vert |f_j|^{r_0}\Vert_{L^{p_j(\cdot)/r_0}(w_j^{r_0})}^{1/r_0}
    =\prod_{j=1}^{m}\Vert f_j\Vert_{L^{p(\cdot)}(w_j)}.
\end{align*}
The rest of the proof proceeds exactly as in the proof of \cite[Theorem 2.3]{CHSW2025}; we omit the details.
\end{proof}

\begin{remark}
According to \cite[Lemma 2.6]{Li_Sun_2013}, the constant $p_0$ can have the following form:
\begin{equation*}
    p_0 = t\cdot \frac{mn}{s},
\end{equation*}
    where the constant $t$ satisfies
\begin{equation*}
    1 < t < \min\left\{t_1,\ldots,t_{m}, \frac{s}{s-1}, \frac{2s}{mn}\right\}
\end{equation*}
for some $1<t_1,\ldots,t_m<\infty$ that can be chosen without any additional restrictions.
\end{remark}

Next, we obtain the following weighted boundedness of the commutators $[T_{\mathfrak{m}},b]_{e_j}$.

\begin{theorem}\label{thm:bound_FourMult_com}
Let $b\in\mathrm{BMO}(\R^n)$. Assume $\frac{mn}{2}<s\leq mn$ and $\mathfrak{m}\in\mathcal{W}^{s}(\R^{nm})$. Consider the constant $p_0\in\left(\frac{mn}{s},\infty\right)$ depending only on $m,n$ and $s$ from Theorem~\ref{thm:bound_FourMult}. Then, the following holds. For all $r_0\in(p_0,\infty)$, for all proper $m$-admissible quadruples $(\vec{p}(\cdot),\vec{p}(\cdot),\vec{r},\infty)$ with $r_j := r_0$, $j = 1, \ldots, m$, and for all $\vec{w}\in\mathcal{A}_{(\vec{p}(\cdot),\vec{p}(\cdot)),(\vec{r},\infty)}$, we have that $[T_{\mathfrak{m}},b]_{e_j}$ maps $L^{p_1(\cdot)}(w_1)\times\cdots\times L^{p_m(\cdot)}(w_m)$ into $L^{p(\cdot)}(\nu_{\vec{w}})$ boundedly, for all $j=1,\ldots,m$.
\end{theorem}

\begin{proof}
Fix $j\in\{1,\ldots,m\}$. By \cite[Lemma 2.7]{Li_Sun_2013} we have that for all bounded compactly supported functions $f_1,\ldots,f_m$ one has
\begin{equation*}
    M^{\#}_{\delta}([T_{\mathfrak{m}},b]_{e_j}(\vec{f})) \lesssim_{\delta,\varepsilon,r_0}\Vert b\Vert_{\mathrm{BMO}(\R^n)}(M_{\varepsilon}(T_{\mathfrak{m}}(\vec{f}))+\mathcal{M}(|f_1|^{r_0},\ldots,|f_m|^{r_0})^{1/r_0}),
\end{equation*}
for all $0<\delta<\varepsilon<\frac{p_0}{m}$ and all $r_0>p_0$.
    
Let us now fix such constants $0<\delta<\varepsilon<\frac{p_0}{m}$. Let $r_0\in(p_0,\infty)$, $(\vec{p}(\cdot),\vec{p}(\cdot),\vec{r},\infty)$ be a proper $m$-admissible quadruple with $r_j := r_0$, $j = 1, \ldots, m$, and let $\vec{w}\in\mathcal{A}_{(\vec{p}(\cdot),\vec{p}(\cdot)),(\vec{r},\infty)}$. Let $f$ be any ``nice'' function. Since $\varepsilon < \frac{r_0}{m}$, Jensen's inequality yields
\begin{equation*}
    M_{\varepsilon}(f) \leq M_{r_0/m}(f).
\end{equation*}
Thus, using Lemma \ref{lem:homog} we estimate
\begin{align*}
    \Vert M_{\varepsilon}(f)\Vert_{L^{p(\cdot)(w)}} &\leq \Vert M_{r_0/m}(f)\Vert_{L^{p(\cdot)(w)}}=\Vert M(|f|^{r_0/m})w^{r_0/m}\Vert^{m/r_0}_{L^{mp(\cdot)/r_0}}\\
    &= \Vert M(|f|^{r_0/m})\Vert^{m/r_0}_{L^{mp(\cdot)/r_0}(w^{r_0/m})}.
\end{align*}
By Lemma~\ref{lem:Apqrs char.} we have
\begin{equation*}
    w \in \mathcal{A}_{(p(\cdot), p(\cdot)),(r_0/m,\infty)},
\end{equation*}
therefore by Lemma~\ref{lem:rescale} we deduce
\begin{equation*}
    w^{r_0/m} \in \mathcal{A}_{(mp(\cdot)/r_0, mp(\cdot)/r_0),((mr_0)/(mr_0),\infty)},
\end{equation*}
that is
\begin{equation*}
    w^{r_0/m} \in \mathcal{A}_{mp(\cdot)/r_0}.
\end{equation*}
Therefore, we obtain
\begin{equation*}
    \Vert M_{r_0/m}(f)\Vert_{L^{p(\cdot)(w)}}
    \lesssim \Vert |f|^{r_0/m}\Vert^{m/r_0}_{L^{mp(\cdot)/r_0}(w^{r_0/m})}
    =\Vert f\Vert_{L^{p(\cdot)}(w)}.
\end{equation*}
Moreover, in the exact same way as in the proof of Theorem~\ref{thm:bound_FourMult} above we have
\begin{equation*}
    \Vert \mathcal{M}(|f_1|^{r_0},\ldots,|f_m|^{r_0})^{1/r_0}\Vert_{L^{p(\cdot)}(w)} \lesssim \prod_{j=1}^{m}\Vert f_j\Vert_{L^{p(\cdot)}(w_j)}.
\end{equation*}
The rest of the proof proceeds exactly as in the proof of \cite[Theorem 2.4]{CHSW2025}; we omit the details.
\end{proof}

\begin{remark}
Both in \cite[Lemma 2.6]{Li_Sun_2013} and \cite[Lemma 2.7]{Li_Sun_2013} the symbol $\mathfrak{m}$ is assumed to satisfy a slightly stronger condition than $\mathfrak{m}\in\mathcal{W}^{s}(\R^{mn})$. However, an examination of the proofs of \cite[Lemma 2.6]{Li_Sun_2013} and \cite[Lemma 2.7]{Li_Sun_2013} shows that the condition $\mathfrak{m}\in\mathcal{W}^{s}(\R^{mn})$ is enough.
\end{remark}

Applying now our second abstract main result, Theorem~\ref{thm:main_result_limited_range}, we obtain the following.

\begin{theorem}\label{thm:compact_FourMult_com}
Let $\vec{b}=(b_1,b_2)$ be a pair of function in $\mathrm{CMO}(\R^n)$. Assume $n<s\leq 2n$ and $\mathfrak{m}\in\mathcal{W}^{s}(\R^{2n})$. Consider the constant $p_0\in\left(\frac{2n}{s},\infty\right)$ depending only on $n$ and $s$ from Theorems~\ref{thm:bound_FourMult} and~\ref{thm:bound_FourMult_com} (for $m=2$). Then, the following holds. For all $r_0\in(p_0,\infty)$, for all proper $2$-admissible quadruples $(\vec{p}(\cdot),\vec{p}(\cdot),\vec{r},\infty)$ with $r_j := r_0$, $j = 1, 2$, and for all $\vec{w}\in\mathcal{A}_{(\vec{p}(\cdot),\vec{p}(\cdot)),(\vec{r},\infty)}$, we have that $[T_{\mathfrak{m}},b_1]_{e_1}$ and $[T_{\mathfrak{m}},b_2]_{e_2}$ map $L^{p_1(\cdot)}(w_1)\times\cdots\times L^{p_m(\cdot)}(w_m)$ into $L^{p(\cdot)}(\nu_{\vec{w}})$ compactly, for all $j=1,2$. In particular, $[T_{\mathfrak{m}},\vec{b}]_{\Sigma}$ maps $L^{p_1(\cdot)}(w_1)\times\dots\times L^{p_m(\cdot)}(w_m)$ into $L^{p(\cdot)}(\nu_{\vec{w}})$ compactly.
\end{theorem}

\begin{proof}
The assertion of the last sentence follows from the previous assertions, since the sum of bilinear compact operators remains compact per \cite[Proposition 2]{BenTor2013}. Thus, we concentrate on proving the rest of the theorem.

Fix $j\in\{1,2\}$, $r_0\in(p_0,\infty)$, a proper 2-admissible quadruple $(\vec{p}(\cdot),\vec{p}(\cdot),\vec{r},\infty)$ with $r_j := r_0$, $j = 1, 2$, and $\vec{w}\in\mathcal{A}_{(\vec{p}(\cdot),\vec{p}(\cdot)),(\vec{r},\infty)}$. To show that $[T_{\mathfrak{m}},b]_{e_j}$ maps $L^{p_1(\cdot)}(w_1)\times\cdots\times L^{p_m(\cdot)}(w_m)$ into $L^{p(\cdot)}(\nu_{\vec{w}})$ compactly, we apply Theorem~\ref{thm:main_result_limited_range}.

Assumption \eqref{eq:main1_limited_range} of Theorem~\ref{thm:main_result_limited_range} for $[T_{\mathfrak{m}},b]_{e_j}$ is just the content of Theorem~\ref{thm:bound_FourMult_com} for $m=2$ (with different notation).
    
We check assumption \eqref{eq:main2_limited_range} of Theorem~\ref{thm:main_result_limited_range}. Set
\begin{equation*}
    t_1 = t_2 := t_0 := \frac{2n}{s}.
\end{equation*}
Then, we have $1\leq t_1, t_2<2$ and
\begin{equation*}
    \frac{1}{t_1}+\frac{1}{t_2} = \frac{s}{n}.
\end{equation*}
Observe that $r_0 > p_0 > t_0\geq1$. Pick $d_0 > \max\{2,r_0\}$.
Then, Theorem~\ref{thm:comp_FourMult_com_unw} yields that part \eqref{eq:main2_limited_range} of the assumptions of Theorem~\ref{thm:main_result_limited_range} is satisfied with the particular choices $\vec{p}_1(\cdot):=(d_0,d_0)$, $p_1(\cdot):=\frac{d_0}{2}$ and $\vec{w}_1:=(1,1)$ for each of the operators $[ T_{m},b_1]_{e_1}$ and $[T_{\mathfrak{m}},b_2]_{e_2}$ .

An application of Theorem~\ref{thm:main_result_limited_range} concludes the proof.
\end{proof}

\printbibliography

\end{document}